\documentclass[psamsfonts, reqno]{amsart}
\usepackage{amssymb,amsfonts}
\usepackage[all,arc]{xy}
\usepackage{enumerate}
\usepackage{mathrsfs}
\usepackage{graphicx}
\usepackage{caption}
\numberwithin{figure}{section}
\newtheorem{thm}{Theorem}[section]

\newtheorem{lem}[thm]{Lemma}
\newtheorem{conj}[thm]{Conjecture}

\theoremstyle{definition}

\theoremstyle{remark}
\newtheorem{rem}[thm]{Remark}

\newcommand{\sF}{{\mathcal F}}
\newcommand{\sG}{{\mathcal G}}

\newcommand{\ZZ}{{\mathbb Z}}

\newcommand{\f}{{f}}
\newcommand{\F}{\overline{F}}
\newcommand{\G}{\overline{G}}
\newcommand{\Kn}{{K_{n}}}
\newcommand{\Ln}{{L_{n}}}
\newcommand{\ord}{{\rm ord}}
\newcommand{\dgt}{{d}}
\newcommand{\Sum}{{S}}

\newcommand{\R}{\overline{R}}
\newcommand{\Gp}{{G}}

\newcommand{\W}{{\nu}}

\numberwithin{equation}{section}

\title{Products of   Farey Fractions}

\author{Jeffrey C. Lagarias}
\address{Department of Mathematics, University of Michigan,
Ann Arbor, MI 48109-1043,USA}
\email{lagarias@umich.edu}

\author{Harsh Mehta}
\address{Department of Mathematics, University of South Carolina,
Columbia, SC 29208}
\email{hmehta@math.sc.edu}

\date{May 7, 2017, corrected}
\thanks{Work of the first author was supported by NSF Grants DMS-1101373 and DMS-1401224.}

\begin{document}

\begin{abstract}
The { Farey fractions} $\sF_n$ of order $n$ consist of all fractions $\frac{h}{k}$i
 lying  in the closed unit interval
and having  denominator at most $n$.
in the unit interval with denominator at most $n$, not necessarily in lowest terms.
 This paper considers  the  products $F_n$ of all nonzero  Farey fractions
of order $n$. It  studies their  growth
 and their divisibility properties  by powers of a fixed prime, given by $\ord_p(F_n)$, as a function of $n$.
It presents evidence suggesting  that information related to the Riemann hypothesis
may be encoded in  functions related to $\ord_p(F_n)$ for a single fixed prime $p$.
This  encoding makes use of a relation of these
products to the products $G_n$ of all reduced and unreduced Farey fractions of order $n$,
which are connected by M\"{o}bius inversion. It introduces new arithmetic functions
which mix the M\"{o}bius function with functions of radix expansions to a fixed prime base $p$.

\end{abstract}

\maketitle

%\tableofcontents

%************************************************************************
%
%  section 1. Introduction
%
%
%************************************************************************
\section{Introduction}

The {\em Farey sequence} $\sF_n$ of order $n$  is the sequence of  reduced fractions $\frac{h}{k}$ between $0$ and $1$ 
(including $0$ and $1$) which, when in lowest terms, have denominators less than or equal to $n$, arranged in order of increasing size.
 We write it  as 
 $$
\sF_n := \{ \frac{h}{k}: 0 \le h \le k \le n: \, gcd(h,k) = 1.\}
$$
Farey sequences are important in studying  Diophantine approximation properties of real numbers, cf. Hardy and Wright \cite[Chap. III]{HW79}. 
They can be viewed as  additive objects that  encode deep arithmetic properties of both integers and the rational numbers.

The set of Farey fractions $\sF_n$ is known to  approximate the  uniform distribution on the
unit interval $[0,1]$ as $n \to \infty$, viewing it as defining a  measure given by a sum of
(normalized) delta functions at the points of $\sF_n$.
The rate at which these  measures   approach the uniform
distribution  can be related to the Riemann hypothesis.
A precise version is given in  a celebrated  theorem of Franel  \cite{Fra24},
with extensions made in many later works, including Landau \cite{Landau27}, 
Mikol\'{a}s \cite{Mik49}, \cite{Mik51},  Huxley \cite[Chap. 9]{Hux72}, 
and Kanemitsu and Yoshimoto \cite{KY96}, \cite{KY00}.

%%%%%%%%%%%%%%%%%%%%
% Subsectection 1.1
%%%%%%%%%%%%%%%%%%
\subsection{Farey products}

We consider a multiplicative statistic associated 
to the Farey fractions--the  products of the nonzero elements of the Farey sequence,
termed { \em Farey products}.
%%%%%%%%%%%%%%%%%%%%%%%%%%%%%%%%%%%%%%%%%%%%%%%
%These products can be viewed as  multiplicative statistics associated to the Farey sequence.
%%%%%%%%%%%%%%%%%%%%%%%%%%%%%%%%%%%%%%%%%%%%%%%
 To study Farey products  we use the {\em positive Farey sequence}
$$
\sF_n^{\ast} :=\sF_n \smallsetminus \{0\} =  \{ \frac{h}{k}: 1 \le h \le k \le n: \, gcd(h,k) = 1\}.
$$
For example, we have
$$
\sF_4^{\ast}:= \left\{ \frac{1}{4}, \frac{1}{3}, \frac{1}{2},  \frac{2}{3}, \frac{3}{4}, \frac{1}{1} \right\}.
$$
We let $\Phi(n) = |\sF_n^{\ast}|$ denote the number of elements of $\sF_n^{\ast}$, and we  clearly have:
\begin{equation}\label{eq-phi}
\Phi(n) = \sum_{k=1}^n \varphi(k),
\end{equation}
where $\varphi(k)$ denotes the Euler totient function, which has
$$
\varphi(k) = | (\ZZ/k\ZZ)^{\times}| = | \{ a:  1 \le a \le k \, \mbox{with} \, \gcd(a, k) =1\} |.
$$
%%%%%%%%%%%%%%%%%%%%%%%%%%%%%%%%%%%%%%%%%%%%%%%%%%%
%Thus $\Phi(4) = \varphi(1) + \varphi(2) + \varphi(3) + \varphi(4) = 1+1+ 2+ 2 =6.$
%%%%%%%%%%%%%%%%%%%%%%%%%%%%%%%%%%%%%%%%%%%%%%%%%%%
To describe the ordered Farey fractions we introduce the notation $\rho_{r} = \rho_{r, n}$ for the
$r$-th fraction in the ordered sequence $\rho_{r, n}< \rho_{r+1, n}$, writing
$$
\sF_n^{\ast} = \{ \rho_{r} = \rho_{r, n}: 1\le r \le \Phi(n)\}.
$$
The product of the Farey fractions is then 
\begin{equation}\label{FPP}
%P_f(n)
F_n :=\prod_{r=1}^{\Phi(n)}\rho_{r,n} = \frac{N_n}{D_n},
\end{equation}
in which $N_n$   denotes the product of the numerators of all the $\rho_{r, n}$ and  $D_n$ the 
product of their denominators;  here $N_n/D_n$ is  not in lowest terms for $n >2$, cf. Section \ref{sec45}.
The {\em Farey product}  $F_n$  is a rational number in the unit interval that rapidly gets small as $n$ increases.

It proves  convenient
 to  introduce
the {\em reciprocal Farey products}
\begin{equation}\label{recip-prod}
\F_n := \frac{1}{F_n} = \frac{D_n}{N_n} = \Big( \prod_{r=1}^{\Phi(n)} \rho_r \Big)^{-1},
\end{equation}
to facilitate comparison with other results (\cite{LM14u}) ; the values of $\F_n$ rapidly increase with $n$.
Clearly $\F_n \ge 1$ and we find that $\F_1 = 1, \F_2= 2, \F_3 =9, \F_4 = 48, \F_5 =1250, \F_6=9000$.
In these examples  $\gcd(N_n, D_n)$ becomes  large,  such  that   $\F_n$ is  an integer for small $n$.
However $\F_7 = \frac{3 \cdot 5^2 \cdot 7^6}{2}$ is not an integer, and it is known that only finitely many $\F_n$
are integers, see Section \ref{sec44}.

Reciprocal Farey products  
have the following  interesting features.
\begin{enumerate}
\item
The statistic $\F_n$ extracts a single rational number from the whole collection of Farey fractions $\sF_n^{\ast}$.
The growth behavior of the numbers $\F_n$ encodes the Riemann hypothesis,
as a consequence of  a 1951 result of Mikol\'{a}s \cite{Mik51}  presented in Section \ref{sec31}.
 This encoding concerns the size of an error term in an approximation
 of $\log \F_n$ in which the main term  is an arithmetic function related to both
 the Euler totient function and the von Mangoldt function. 
\item
The  functions $\W_p(\F_n)=\ord_p(\F_n)$ that describe divisibility of $\F_n$
by a (positive or negative) power of a fixed prime $p$,
have an interesting structure. 
Here  $\ord_p(\F_n)$ gives  the exact (positive or negative or zero) power of $p$ dividing $\F_n$,
so that $p^{- \ord_p (\F_n)} \F_n$ is a rational number having both numerator and denominator prime to $p$,
 and $||\F_n||_p = p^{- \ord_p(\F_n)}$ is the usual $p$-adic valuation of $\F_n$.
There is generally a large cancellation of powers of $p$ in the numerator
and denominator of the product defining $\F_n$, and the behavior of this cancellation is of interest.
\end{enumerate} 
Since reciprocal Farey products   encode the Riemann hypothesis we
may expect in advance that they will exhibit   complicated and mysterious arithmetic behavior.
Even  simple-looking questions may prove to be quite difficult.

%%%%%%%%%%%%%%%%%%%%
% Subsectction 1.2
%%%%%%%%%%%%%%%%%%
\subsection{Results}\label{sec12}
We study the size of the rational numbers $\F_n$ at the real place measured
using a logarithmic scale by
\begin{equation}\label{F-infty}
\W_{\infty}(\F_n) := \log(\F_n).
\end{equation}
For each prime $p$, we study the functions
\begin{equation} \label{F-prime}
\W_p(\F_n) :=\ord_p(\F_n)
\end{equation}
 which measure 
the $p$-divisibility of $\F_n$; the values $\ord_p(\F_n)$ may be positive or negative.

The investigations of this paper first
obtain information  on   Farey products $F_n$ as they relate to  
 the products  of all
reduced and unreduced Farey fractions $G_n$, which we term
{\em unreduced Farey products}. 
The  {\em reciprocal unreduced Farey products} $\G_n = 1/G_n$ are always integers,  
equal  to the product of all binomial coefficients in the $n$-row
of Pascal's triangle.   

In Section \ref{sec2} we study the  reciprocal unreduced Farey products $\G_n$,
first summarizing some  results taken from our paper \cite{LM14u}.
The  function $\log(\G_n)$ has a smooth growth given by an asymptotic expansion valid
to all orders of $\frac{1}{n^k}$. The  functions $\ord_p(\G_n)$ have a complicated but analyzable behavior related
 to the base $p$ radix expansions of the integers from $1$ to $n$. One also has
 $0 \le \ord_p(\G_n) < n \log_p n$. Then in  Section \ref{sec22} we give  basic relations between $\F_n$
 and $\G_n$ which involve the floor function. These  start with the product relation 
 $$\G_n =\prod_{\ell=1}^n \F_{[n/ \ell]},$$
and by  M\"{o}bius inversion we obtain  the basic identity
$$\F_n= \prod_{\ell=1}^n (\G_{[n/\ell]})^{\mu(\ell)}.$$
We obtain further formulas by splitting the sums using a parameter $L$ related to 
the Dirichlet hyperbola method, as formulated in Diamond \cite[Lemma 3.1]{Dia82}.
\medskip

In Section \ref{sec3} we turn to $\F_n$ and study  the growth rate of $\log(\F_n)$.
This function does not have a complete asymptotic expansion in terms of simple functions.
We review known results of Mikol\'{a}s which
relate fluctuations of this growth rate to the Riemann hypothesis. They say  that $\log(\F_n)$ is well approximated
by a main term  $\Phi(n) - \frac{1}{2} \psi (n)$, in which $\Phi(n)$ is as defined in \eqref{eq-phi},
$\psi(n) = \sum_{k \le n} \Lambda(k)$,
with $\Lambda(n)$ being the von Mangoldt function. The size of the remainder term
 $R_{\F}(n) =\log(\F_n) -\big(\Phi(n) - \frac{1}{2}\psi(n)\big)$ is
then  related to the Riemann hypothesis. In Section \ref{sec3}  we also review known results about the 
fluctuating behavior of $\Phi(n)$.

In Section \ref{sec4a} we study the functions $\ord_p(\F_n)$. 
These functions  have a more complicated behavior than  of $\ord_p(\G_n)$.  We give  formulas
for computing $\ord_p(\F_n)$, and present  experimental data 
on  its values for small primes $p$.
We do not understand the behavior of $\ord_p(\F_n)$ well theoretically, and
our data leads us to formulate a set of four hypotheses stating (unproved) properties
(P1) - (P4) that these functions might have. These hypothetical properties  (P1) - (P4)  include assertions that 
$\ord_p(\F_n)$ has infinitely many sign changes; that  a sign change
always occurs between $n =p^k-1$ and $n = p^k$, for $k>1$;  and that the growth rate of $\ord_p(\F_n)$
is of order $O( n \log_p n)$. 
Even very special cases of these properties are unsolved problems
which may be hard. For example: Is it true that 
for a prime $p$  the inequality 
$\ord_p(\F_{p^2 -1}) \le 0$ always holds?   
This assertion comprises a  family of one-sided inequalities involving M\"{o}bius function sums.
A family of one-sided inequalities of this sort, if true, would be of great
interest as providing fundamental new arithmetic information about the M\"{o}bius function.
At the end of Section \ref{sec4a} we present a few  theoretical results supporting
the possible validity of these properties.  We show that for each $p$ there is at least one sign change 
in the value of $\ord_p(\F_n)$.
Concerning the size of $\ord_p(\F_n)$ we have  the easy bound $|\ord_p(\F_n)|\le n (\log_p n)^2$
which follows from knowledge of $\ord_p(\G_n)$.

 In Section \ref{sec5} we study relations between the growth rate of $\log(\F_n)$ and the Riemann hypothesis,
 given by the result of Mikol\'{a}s, with
  the  main term in Mikol\'{a}s formula being $ \Phi(n) - \frac{1}{2} \psi(n)$.
 In this section we relate this main term to
 a quantity  given entirely in  terms of  $\log (\G_n)$ and the M\"{o}bius function, using a parallel with 
  the ``hyperbola method" of Dirichlet.  In Section \ref{sec52} we justify our definition of ``replacement main term" $\Phi_{\infty}(n)$
by showing that the Riemann hypothesis implies that it is indeed close to the ``main term" $\Phi(n) - \frac{1}{2}\psi (n)$
in the Mikol\'{a}s formulation of the Riemann hypothesis. 
We obtain a formula for the ``replacement remainder term" and present empirical evidence about its behavior.
It has a very striking non-random features in which the influence of the M\"{o}bius function is clearly visible.

In  Section \ref{sec6}, we  ask: 
{\em Can one  approach the Riemann
 hypothesis through knowledge  of the function $\ord_p(\F_n)$ at a single fixed prime $p$?}
  Note that the product formula for rational numbers expresses  $\log(\F_n)$  as
 a  weighted sum of $\ord_p(\F_n)$ for $p \le n$, and by the Mikol\'{a}s result
 this in principle  allows the  Riemann hypothesis 
 to be expressed as a complicated function of  {\em all}  the functions $\ord_p(\F_n)$
 with variable $p$. Speculation  that the Riemann hypothesis  might be visible from data at a {\em single} prime $p$  seems initially  
 unbelievable.
 It becomes less far-fetched when one observes 
 from the formulas  that the full set of M\"{o}bius function values $\{ \mu(m) : n\geq m \geq 1\}$  influence 
 the values $\ord_p(\F_n)$. 
 
 Section \ref{sec6} parallels the recipe of Section \ref{sec5} 
 in  formulating at the prime $p$  formulas analogous to the ``replacement main term", given
now in terms of $\ord_p(\G_n)$, which might serve as a  ``main term" to approximate the function
$\ord_p(\F_n)$.
The new ``replacement main terms" and ``remainder terms"
 are based on the M\"{o}bius inversion relation between $\ord_p(\F_n)$ and  the $\ord_p(\G_n)$,
and the resulting division into two terms  is related to the Dirichlet hyperbola method.
For a fixed $p$ there
 are now three different possible recipes 
to split off a ``main term" and ``remainder term", unlike the archimedean case
considered in Section \ref{sec5}.
The resulting terms  include new kinds of arithmetic sums  not studied before: 
{\em  individual terms   in these sums involve  M\"{o}bius function values multiplied by 
sums of the base $p$ digits at selected integer values.} These new ``replacement main
terms" themselves have unusual structure, being oscillatory functions. However, 
after they are removed, 
one can   ask the question whether the size of the ``remainder terms" in these new expressions is
related to zeta zeros, and in particular to the Riemann hypothesis. 

 We try  all three  for a ``replacement main term",
and find  experimentally that one of them gives plots of
the remainder term having  non-random features in striking
parallel with the experimental data in the archimedean case in Section \ref{sec5}. 
This observation was a  remarkable experimental discovery of this work.

In the final Section \ref{sec7} we make concluding remarks
on this possible encoding of the Riemann hypothesis at  a fixed
finite prime.

%%%% 
%This experimental evidence seems a  most interesting discovery.
%%%%%

In Appendix A (\ref{secAA}
 we present additional computational results for $p=3$
complementing results for $p=2$ given in Section \ref{sec43}.

%************************************************************************
%
%  section 2
%
%
%************************************************************************

\section{Unreduced Farey Products }\label{sec2}

Unreduced Farey products  provide  an approach to 
 understand the   Farey products. 
  The  {\em unreduced Farey sequence} $\sG_n$ is 
the   ordered sequence of all reduced and unreduced fractions 
between $0$ and $1$ with denominator of size at most $n$, and its positive analogue,
which we denote
$$
\sG_n^{\ast}  := \left\{ \frac{h}{k}: 1 \le h \le k \le n\right\}.
$$
We order these unreduced fractions in increasing order, breaking ties between equal fractions 
by placing them in  order of increasing denominator. For example, we have
$$
\sG_4^{\ast}:= \left\{ \frac{1}{4}, \frac{1}{3}, \frac{1}{2}, \frac{2}{4}, \frac{2}{3}, \frac{3}{4}, \frac{1}{1}, \frac{2}{2}, \frac{3}{3}, \frac{4}{4} \right\}.
$$
Denoting the  number of elements in $\sG_n^{\ast}$ as $\Phi^{\ast}(n)$, we may 
may label the fractions in $\sG_n^{\ast}$ in order as $\rho_r^{*} = \rho_{r,n}^{\ast}$ and write
$$
\sG_n^{\ast} = \{ \rho_{r}^{*}= \rho_{r, n}^{*}: 1\le r \le \Phi^{\ast}(n)\}.
$$
Here we have 
\begin{equation} \label{unred-tot}
\Phi^{\ast}(n) = \sum_{k=1}^n k = {{n+1}\choose {2}}.
\end{equation}
Now we define the {\em unreduced Farey product} 
$$
G_n :=\prod_{r=1}^{\Phi^{\ast}(n)} \rho_{r,n}^{\ast}= \frac{N_n^{\ast}}{D_n^{\ast}},
$$
where  $N_n^{\ast}$  denotes the product of the numerators of all $\rho_{r,n}^{\ast}$ and $D_n^{\ast}$
the corresponding product of  denominators; certainly   $N_n^{\ast}/D_n^{\ast}$ is not in lowest terms.
Now we define the {\em reciprocal unreduced  Farey product} 
$$
\G_n:= 
\frac{1}{G_n} =\frac{D_n^{\ast}}{N_n^{\ast}}= \Big( \prod_{r=1}^{\Phi_n^{\ast}} \rho_r \Big)^{-1}.
$$ 
Here $\G_1=1, \G_2=2, \G_3=9, \G_4=96, \G_5=2500, \G_6=162000$ and $\G_7= 3^2 \cdot 5^2 \cdot 7^6$
is an integer.

%************************************************************************
%
%  section 2;1 
%
%
%************************************************************************

\subsection{Properties of reciprocal unreduced Farey products $\G_n$}\label{sec21}

This section recalls results from 
a  detailed study of reciprocal unreduced Farey products $\G_n$
made in \cite{LM14u}.
from which we recall the following results.  A first result is that
$\G_n$ is always an integer, being given as a product of binomial
coefficients
\begin{equation}\label{2000}
\G_n= \prod_{j=0}^n {{n}\choose{j}}.
\end{equation}
For this reason the $\G_n$ are called {\em binomial products} in \cite[Theorem 2.1]{LM14u}.
The numerators and denominators in  this  formula 
have asymptotic expansions which when combined yield a  good asymptotic expansion for $\log (\G_n)$, 
valid when  $n$ is a positive  integer (\cite[Theorem 3.1, Appendix A]{LM14u}).

%%%%%%%%%%%%%%%%%%%%%%%%%%%
% Thm 2.1, asymptotic formula
%%%%%%%%%%%%%%%%%%%%%%%%%%%
\begin{thm}\label{th20}
For  positive integers $n \to \infty$ there holds 
\begin{equation}\label{asymp}
\log  (\G_n)=\frac{1}{2}n^2-\frac{1}{2}n\log n+(1-\log(\sqrt{2\pi}))\, n -\frac{1}{3}\log n+g_0+ O(\frac{1}{n}).
\end{equation}
In this formula  $g_0=  -\frac{1}{2} \log (2 \pi)  - \frac{1}{12}+ 2 \log A$ with $A$
denoting the Glaisher-Kinkelin constant
$A=\exp \big( \frac{1}{12}- \zeta^{'}(-1)\big) \approx 1.282427$.
\end{thm}

One may extend   $\G_n$ to a function of a real variable $x$
as a  step function $\G_x := \G_{\lfloor x \rfloor}$.
When this is done, 
the asymptotic expansion \eqref{asymp} above  remains valid  {\em only at integer values of $x$}; the jumps in the step function
are of size $\ge n$, which is larger than all but the first three terms in the expansion \eqref{asymp}. For later
use, we restate  this in the form
\begin{equation}\label{asymp2}
\log  (\G_n)=\Phi^{\ast}(n) -\frac{1}{2}n\log n+\frac{1}{2}\log(\frac{e}{2\pi})\, n + O (\log n).
\end{equation}

Secondly  we have essentially sharp upper and lower bounds for $\ord_p(\G_n)$ 
(\cite[Theorems 6.7 and  6.8]{LM14u}).

%%%%%%%%%%%%%%%%%%%%%%%%%%%
% Thm 2.2, hard lower and upper bounds 
%%%%%%%%%%%%%%%%%%%%%%%%%%%
\begin{thm}\label{th29}
For each prime $p$, there holds for all $n \ge 1$, 
\begin{equation}\label{eq216}
0 \le \ord_p(\G_n) < n \log_p n.
\end{equation}
The value at $n= p^k-1$ is  $\ord_p(G_{p^k-1})=0$. 
The value at  $n=p^k$ is
\begin{equation}\label{sharp-bd}
\ord_p(\G_{p^k})=\left(kp^k-\frac{p^k-1}{p-1}\right).
\end{equation}
This value has  $\ord_p(\G_n) \ge n \log_p n - n$. 
\end{thm}

We record next an explicit formula for $\ord_p(\G_n)$,
which is related to the base $p$ expansion of $n$.
We write a positive integer $n$ in a general radix base $b \ge 2$ as
$$
n := \sum_{i=0}^{k} a_i b^i, \, \mbox{for} \,  b^k \le n < b^{k+1}.
$$
with digits $0 \le a_i=a_i(n) \le b-1$ and $k = \lfloor \log_b n\rfloor.$

The {\em sum of digits function (to base $b$)} of $n$  is
\begin{equation}\label{sum-dig}
\dgt_b(n) := \sum_{i=0}^k a_i(n),
\end{equation}

 The {\em total digit summatory  function (to base $b$)}  is
\begin{equation}\label{tot-sum-dig}
\Sum_b(n) := \sum_{j=0}^{n-1} \dgt_b(j).
\end{equation}

Then we have (\cite[Theorem 5.1]{LM14u})

%%%%%%%%%%%%%%%%%%%%
% Thm 2.3, closed exact digit formula
%%%%%%%%%%%%%%%%%%%%
\begin{thm}\label{th39} 
Let the prime $p$ be fixed. Then for all $n \ge 1$, 
\begin{equation}\label{summatory}
\W_p(\G_n) = \ord_p (\G_n) = \frac{1}{p-1} \Big(2\Sum_p(n) - (n-1) \dgt_p(n)   \Big).
\end{equation}
%%%%%%%%%%%%%%%%%%%%%%%%%%%%%%%%%%%%%%%%%%%%%%%%%%%%%%%%%%%
%%  in which $\dgt_p(n)$ denotes the sum of the base $p$ digits of $n$, and  $\Sum_p(n) = \sum_{j=1}^n \dgt_p(n).$
%%%%%%%%%%%%%%%%%%%%%%%%%%%%%%%%%%%%%%%%%%%%%%%%%%%%%%%%%%%
\end{thm}

This identity was established  starting from an observation made in Granville \cite[equation (18)]{Gra97}.
There is  an explicit expression for $\Sum_p(n)$ due to Delange \cite{Del75}, 
which applies more generally
to radix expansions to an arbitrary integer base $b \ge 2$.

%%%%%%%%%%%%%%%%%
% Theorem 2.4 Delange's thm
%%%%%%%%%%%%%%%%%

\begin{thm}\label{th313} {\em (Delange (1975))}
Given an integer base $b \ge 2$, there exists a function $\f_b(x)$ on the
real line, which is continuous and periodic of period $1$,
such that for all integers $n \ge 1$,
\begin{equation} \label{279}
 \Sum_b(n) = \left(\frac{b-1}{ 2}\right) n \log_b n + \f_b( \log_b n) n.
\end{equation}
\end{thm}

Delange showed that the  function $\f_b(x)$ has a Fourier series expansion
$$
\f_b(x) = \sum_{k \in \ZZ}  c_b(k) e^{2 \pi i k x}
$$
 whose Fourier coefficients
are given for $k \ne 0$ by
\begin{equation}\label{zetac}
c_b(k) = -\frac{ b-1}{2 k \pi i} \left( 1+ \frac{2 k \pi i}{\log b} \right)^{-1} \zeta\left( \frac{2 k \pi i}{\log b}\right).
\end{equation}
with $\zeta(s)$ being the Riemann zeta function,
and  with constant term
\begin{equation}\label{CT}
c_b(0) = \frac{b-1}{2 \log(b)}( \log(2 \pi)-1) - \left(\frac{b+1}{4}\right).
\end{equation}
The function $\f_b(x)$ is continuous but Delange \cite[Sect. 3]{Del75} showed  it
is everywhere non-differentiable, see also Tenenbaum \cite{Ten97}.

To illustrate the behavior of $\ord_p(\G_n)$ which is described by
Theorem \ref{th39} and \ref{th313}, in Figure \ref{fig21-ord2} we give a plot of $\ord_2(\G_n)$
for $1 \le n \le 1023= 2^{10}-1$. The  visible ``streaks" in the plot 
represent values where $\dgt_p(n)=j$ has a constant value.
There are  large jumps in $\ord_2(\G_n)$ between the value $n=2^k-1$
where $\ord_2(\G_n)=0$, and $n=2^k$, where $\ord_2(\G_n) \ge (k-1)n.$

%%%%%%%%%%%%%
% Figure 1
%%%%%%%%%%%%%%
%Figure 2.1
%%%%%%%%%%%%%%%

\begin{figure}[!htb]
\includegraphics[width=135mm]{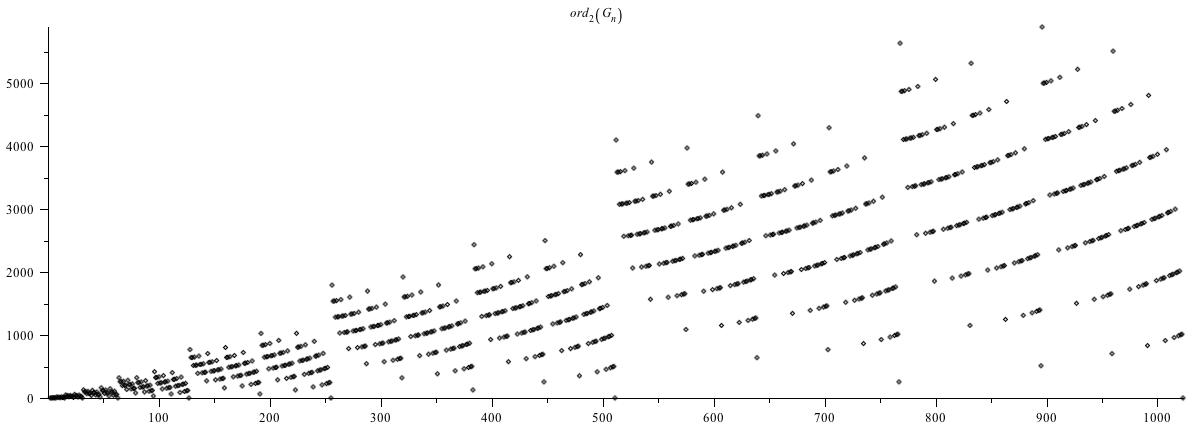}
\caption{$\ord_2(\G_n)$, $1 \le n \le 1023$.}
\label{fig21-ord2}
\end{figure}

In Figure \ref{fig32-CLT} we  plot the behavior of $\frac{1}{n} \ord_2(\G_n)$
over the range of a single power of $2$,  $2^k \le n \le 2^{k+1} -1$, 
 sorted in order of increasing size. 
 The sorted values of $\frac{1}{n} \ord_2(\G_n)$ over this range  have mean about $\frac{1}{2} k = \frac{1}{2} \log_2 n$
as $k \to \infty$, have variance proportional to $\sqrt{\log n})$, and if properly scaled satisfies a central limit
theorem as $k \to \infty$.

%%%%%%%%%%%%%%
%Figure 2.2
%%%%%%%%%%%%%%%

\begin{figure}[!htb]
\includegraphics[width=130mm]{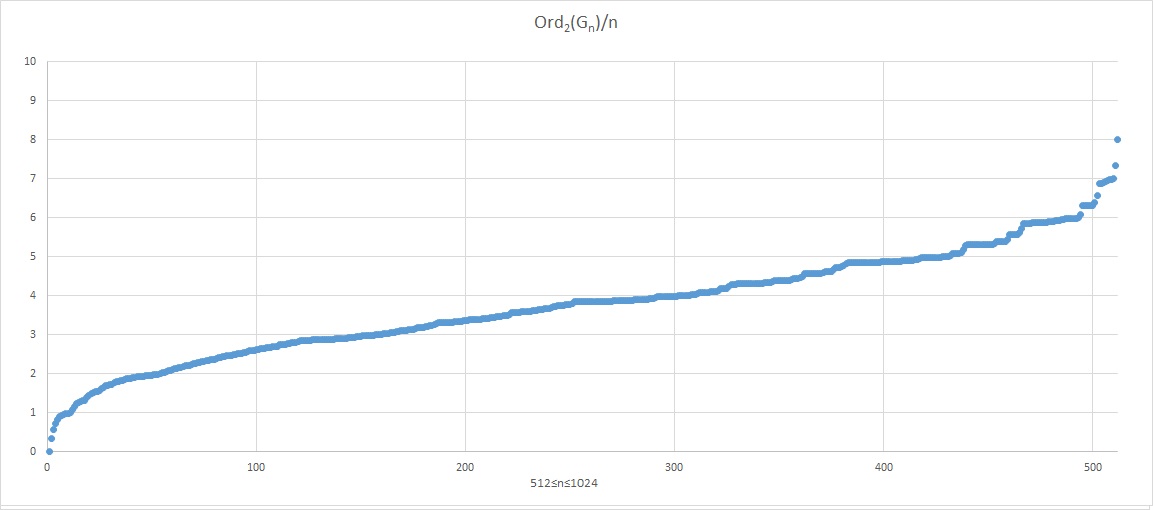}

\caption{Sorted values of $\frac{1}{n} \ord_2(\G_n)$, $512 \le n \le 1024$.}
\label{fig32-CLT}
\end{figure}

These two plots of $\ord_2(\G_n)$ are presented for later comparison with $\ord_2(\F_n)$.

%************************************************************************
%
%  section 2.2 Mobius inversion Key formula Farey Products
%
%
%**********************************************************************
\subsection{Relation of the $\F_n$ and $\G_n$: M\"{o}bius inversion}\label{sec22}

The reciprocal Farey products $\F_n$ are directly expressible  in terms
of  reciprocal unreduced  Farey products $\G_n$ introduced in \cite{LM14u} by M\"{o}bius inversion. 

%%%%%%%%%%%%%%
 % Theorem 2.5
 %%%%%%%%%%%%%% 
\begin{thm}\label{th41}
The  reciprocal unreduced Farey products are related to the reciprocal Farey products by the identity
\begin{equation}\label{F-to-G}
\G_n = \prod_{\ell=1}^n \F_{\lfloor n/\ell \rfloor}.
\end{equation}
 By M\"{o}bius inversion, there holds
\begin{equation} \label{G-to-F}
\F_n = \prod_{\ell=1}^n \big(\G_{\lfloor n/\ell\rfloor}\big)^{\mu(\ell)}.
\end{equation}
%%%%%%%%%%%%%%%%%%%%%%%%%%%%%%%%%%%%%%%%%%%%%%%%%%%%%%%
%In these formulas $\G_x= G_{\lfloor x\rfloor}$ and $\F_x = \F_{\lfloor x\rfloor}$ are regarded as step functions
%of a real variable $x$ to evaluate non-integer $x= n/\ell$.
%%%%%%%%%%%%%%%%%%%%%%%%%%%%%%%%%%%%%%%%%%%%%%%%%%%%%%%
\end{thm}

\begin{proof}
We group the elements $\frac{h}{k}$of $\sG_n^{\ast}$  according
to the value 
$\ell:=\gcd(h, k)$. The fractions with a fixed $\ell$ are in one-to-one correspondence
with elements of the  Farey sequence $\sF_{n/\ell}$, and their product is identical with
the product of the elements of that Farey sequence. This gives the first formula.

To obtain the second formula, we make a detour 
by taking a logarithm to obtain an additive formula, namely
\begin{equation} \label{eq41aa}
\log (\G_n)=\sum_{\ell=1}^n  \,\log ( \F_{\lfloor n/\ell} \rfloor),
\end{equation}
A  variant  of  the M\"{o}bius inversion formula (\cite[Sec. I.2, Theorem 9]{Ten95})  then yields 
\begin{equation} \label{eq41}
\log (\F_n)=\sum_{\ell=1}^n \mu(\ell) \,\log (\G_{\lfloor n/\ell \rfloor}).
\end{equation}
The second formula follows by exponentiating both sides of \eqref{eq41}.
\end{proof}

%%%%
% Remark 2.6
%%%%
\begin{rem}
If we  define $\G_x= \G_{\lfloor x\rfloor}$ and $\F_x = \F_{\lfloor x\rfloor}$  as step functions
of a real variable $x$  then we can rewrite the formulas
above without the floor function notation, as
$$
\G_n = \prod_{\ell=1}^n \F_{ n/\ell} \quad \mbox{and} \quad  \F_n = \prod_{\ell=1}^n \big(\G_{ n/\ell}\big)^{\mu(\ell)}.
$$
However the subtleties in the behavior of these functions certainly has to do with the floor function,
and we prefer to have it visible.
\end{rem}

In the formulas of Theorem \ref{th41} 
 the fractions $\lfloor n/\ell \rfloor$ take only about
$2 \sqrt{n}$ distinct values. 
 This allows the possibility  to  combine terms in the sum and  take advantage of
cancellation in sums of the M\"{o}bius function. We recall that   the {\em Mertens function}
 $M(n)$ is defined by
\begin{equation}
M(n) := \sum_{j=1}^n \mu(j).
\end{equation}
 We split the sum \eqref{eq41} for $\log (\F_n)$ into two parts, using a parameter $L$, as
 \medskip
\begin{eqnarray*}
\log (\F_n) & = & \sum_{k=1}^{n/(L+1)} \mu(k) \log (\G_{\lfloor n/k \rfloor})  + 
\sum_{\ell=1}^L  \left( \sum_{\frac{n}{\ell +1} < k \le \frac{n}{\ell}} \mu(k) \right) \log (\G_{\ell}) \, \nonumber \\
& = & \sum_{k=1}^{n/(L+1)} \mu(k) \log (\G_{\lfloor n/k \rfloor})  + 
\sum_{\ell=1}^L  \left(M\left(\frac{n}{\ell}\right) - M\left(\frac{n}{\ell+1}\right)\right) \log (\G_{\ell}).
\end{eqnarray*}
The second term accumulates  cancellations among consecutive M\"{o}bius function
values. This sort of splitting formula is associated with the Dirichlet hyperbola method,
as formulated in Diamond \cite[Lemma 2.1]{Dia82}, cf.  Tenenbaum \cite[Sect. 3.2]{Ten95}.

 The most balanced parameter choice is $L= \lfloor \sqrt{n} \rfloor$, in which case we  write
\begin{equation}\label{archimedean}
\log (\F_n) = \Phi_{\infty}^{+}(n) + \Phi_{\infty}^{-}(n), 
\end{equation}
setting
\begin{equation}
\Phi_{\infty}^{+}(n) := \sum_{ k =1}^ {n/(\lfloor \sqrt{n} \rfloor +1)} 
%{}^{'} 
\mu(k) \log (\G_{\lfloor n/k \rfloor})
\end{equation}
and
\begin{equation}
\Phi_{\infty}^{-}(n) := \sum_{\ell=1}^{\lfloor \sqrt{n}\rfloor}  \left(M\left(\frac{n}{\ell}\right) - M(n/(\ell +1))\right) \log (\G_{\ell}).
\end{equation}
Note that if  $m^2 \le n < (m+1)^2$ then  
$$
\frac{n}{(\lfloor \sqrt{n}  \rfloor +1)}=
\begin{cases}
\lfloor \sqrt{n} \rfloor -1 & \quad \mbox{if} \quad m^2 \le n < m(m+1)\\
\lfloor \sqrt{n} \rfloor  & \quad \mbox{if} \quad m(m+1) \le n < (m+1)^2.
\end{cases}
$$

It is well known that the Riemann hypothesis is equivalent to the growth estimate
$M(n) = O ( n^{\frac{1}{2} + \epsilon})$ being valid for each $\epsilon>0,$
see Titchmarsh \cite[Theorem 14.25 (C)]{TH86}.
To obtain some unconditional cancellations in the second
sum, one may take
$L$ to be much smaller, e.g. $L =\exp( (\log n)^{\theta})$ for a suitable choice of $\theta$,
and use the unconditional estimate 
$M(n) =   O (n\exp( -(\log n)^{\theta}))$ known to be valid for $\theta> \frac{3}{5}$.
To extract information from the resulting formulas seems to require additional ideas, 
which  we hope to address on another occasion.

%%%%%%%%%%%%%%%%%%%%%%%%%%%
%Remark 2.7
%%%%%%%%%%%%%%%%%%%%%%%%%%%
\begin{rem}
Parallel to \eqref{eq41aa} and \eqref{eq41} for $\log \G_n$ and
$\log \F_n$ there are analogous formulas for prime divisibility of $\G_n$ and $\F_n$.
Applying $\ord_p (\cdot)$ in Theorem \ref{th41} yields 
\begin{equation} \label{eq42aa}
\ord_p (\G_n) =\sum_{\ell=1}^n   \ord_p (\F_{\lfloor n/\ell \rfloor}),
\end{equation}
and
\begin{equation}\label{eq42}
\ord_p ( \F_n) = \sum_{\ell=1}^n \mu(\ell) \,\ord_p (\G_{\lfloor n/\ell \rfloor}).
\end{equation}
We can also split these sums into two parts using a parameter $L$, as done
above.
\end{rem}

%************************************************************************
%
%  section 3
%
%
%************************************************************************

\section{Reciprocal Farey Product $\F_n$ Archimedean Growth Rate}\label{sec3}

The growth rate of Farey  products measured by $\log (\F_n)$ 
was studied by
Mikol\'{a}s  \cite{Mik51}, who showed their behavior encodes 
the Riemann hypothesis.
We describe this result and other known results about its oscillatory main term.

%************************************************************************
%
%  section 3.1 Inverse Farey Products Growth Rate
%
%
%**********************************************************************

\subsection{Mikol\'{a}s's theorem}\label{sec31}

In 1951 Mikol\'{a}s obtained  an asymptotic formula for  the growth rate of $\log(\F_n)$
having an error term related to the  Riemann hypothesis.
To formulate  his
results, we  first recall that, for $Re(s) >1$, there holds
\begin{equation}\label{eq61}
\frac{\zeta'(s)}{\zeta(s)} = - \sum_{n=1}^{\infty} \Lambda(n) n^{-s},
\end{equation}
where the {\em von Mangoldt function} $\Lambda(n)$ has 
$$
\Lambda(n) = 
 \left\{  \begin{array} {ll}  
   \log \, p & \mbox{if} \, \,\, n = p^k, \\%\mbox{a prime power,}\\
    0 &  \mbox{if} \, \,\, n \ne p^k.\, 
    %\mbox{otherwise} 
    \end{array}
    \right.
$$
We define the summatory function
$$
\psi(x) := \sum_{k=1}^{\lfloor x \rfloor} \Lambda (k)
$$
The  prime number theorem with error term states that 
$$
\psi(x)=  x+ O \big( x \exp( - C (\log x)^{\theta})\big)
$$
where the current best exponent is   $\theta = \frac{3}{5} + \epsilon$.

 Mikol\'{a}s \cite[Theorem 1]{Mik51} established the following result,
 showing that 
 $\log(\F_n)$ is well approximated by $ \Phi(n) - \frac{1}{2} \psi(n)$.
 
%%%%%%%%%%%%%%
 % Theorem 3.1
 %%%%%%%%%%%%%% 
\begin{thm} \label{th43}  {\em (Mikol\'{a}s (1951))}
Define the remainder term $R_{\F}(n)$ by the equation
\begin{equation}\label{eq65}
\log (\F_n) =  \Phi(n) - \frac{1}{2} \psi(n)  + R_{\F}(n)
\end{equation}
Then $R_{\F}(n)$ satisfies  the following  bounds.

(1) Unconditionally, there is a constant $C>0$ such that 
$$
| R_{\F}(n)| = O \big( n \exp( - C \sqrt{\log n})\big)
$$
holds for $2 \le n < \infty$.

(2) The Riemann hypothesis is true if and only if, for each $\epsilon>0$,
$$
| R_{\F}(n) |  = O \big( n^{\frac{1}{2} + \epsilon}\big)
$$
holds for $2 \le n < \infty$.
\end{thm}

\begin{proof}
The  remainder term bounds in results (1) and (2) parallel those for 
bounding  $R(x) :=\psi(x)-x$ given above.
Mikol\'{a}s's results are proved for $\log (F_n) = - \log (\F_n).$
Result (1) appears as  Theorem 1 of \cite{Mik51}. 
Result (2) appears as Theorem 2 of \cite{Mik51},
where the constant in the $O$-notation depends on $\epsilon$. 
His result also states that
the Riemann hypothesis implies  the stronger error term
$$
R_{\F}(n) = O\left( \sqrt{n} \exp\left( c \frac{(\log n)( \log\log\log n)}{\log\log n}\right)\right),
$$
valid for $n \ge 50$.
\end{proof}

%%%%%%%%%
%Remark 3.2
%%%%%%%%%
\begin{rem} 
Since  there are exactly $\Phi(n)$ nonzero Farey fractions, Theorem \ref{th44} (1) shows that
from  the
viewpoint  of multiplication the average size of a Farey fraction (i.e. the geometric mean)
is asymptotically $\frac{1}{e}$ as $n \to \infty$.
\end{rem}

%************************************************************************
%
%  section 3.2 Oscillatory Main Term
%
%
%**********************************************************************

\subsection{Behavior of $\Phi(x)$ and $\psi(x)$}\label{sec32}

The encoding of the Riemann hypothesis in Theorem \ref{th43}
requires the inclusion of the oscillatory main term
$\Phi(x) - \frac{1}{2} \psi (x)$, whose fluctuations  
appear to lack  a simple description.

For $\psi(x)$ we have 
$$
\psi(x) = \frac{1}{2 \pi i} \int_{ c- i \infty}^{c + i \infty} -\frac{\zeta'(s)}{\zeta(s)} x^{-s} ds.
$$
The oscillations in $\psi(x)$ around $x$  are directly expressed in terms of the zeta zeros
 by Riemann's explicit formula. 
It is well known (Tenenbaum \cite[Sec. II.4.3]{Ten95}) that the  Riemann hypothesis is  equivalent to the assertion that
$$
\psi(x)  =  x+ O ( x^{\frac{1}{2} + \epsilon}),
$$
holds for each $\epsilon>0$ with a  constant in the $O$-notation 
that depends  on $\epsilon$.
Under the Riemann hypothesis, in  view of the above equation
the term $\psi(n)$ in \eqref{eq65} could be replaced by $n$ and the rest absorbed into the remainder term.

The   function $\Phi(x)$
which counts the number of positive Farey fractions of order $\lfloor x \rfloor$ is
\begin{equation}\label{RH1}
\Phi(x) := \sum_{k=1}^{\lfloor x \rfloor} \varphi (k),
\end{equation}
and can also be obtained by an inverse Mellin transform
$$
\Phi(x) = \frac{1}{2 \pi i} \int_{ c- i \infty}^{c + i \infty} \frac{\zeta(s-1)}{\zeta(s)} x^{-s} ds,
$$
valid for non-integer $x$.
Contour integral methods using this formula can extract the main term
$\frac{3}{\pi^2} x^2$ coming from the simple pole at $s=2$ of $\frac{\zeta(s-1)}{\zeta(s)}$.
It  is difficult to estimate the remainder term $E(x)$, which 
we define by
\begin{equation}\label{remainder}
E(x) := \Phi(x) - \frac{3}{\pi^2} x^2,
\end{equation}
There is a well-known estimate due to Mertens \cite[Sect. 1]{Mer1874},
\begin{equation}\label{phi-bound}
\Phi(x) = \frac{3}{\pi^2} x^2 + O( x \log x),
\end{equation}
see  Hardy and Wright \cite[Theorem 330]{HW79}.
 The current best upper bound on its size was given
in 1962  in A. Walfisz \cite[Chap. IV]{Wal62}, stating  that 
$$
E(x) =  O \big( x (\log x)^{\frac{2}{3}} (\log\log x)^{\frac{4}{3}}\big).
$$
It is also known that $E(x)$ has large oscillations,
with the current  best lower bound on the  size of the fluctuations of $E(x)$
being a 1987 
result of Montgomery \cite[Theorem 2]{Mon87}, stating that\footnote{Here $f(x) = \Omega_{\pm}(g(x))$ means
there is a positive constant such  that infinitely often
$f(x) > c |g(x)|$ and infinitely often $f(x) < - c |g(x)|$.}
$$
E(x) = \Omega_{\pm}( x \sqrt{\log\log x}).
$$
Montgomery formulated the following conjectures concerning
the order of magnitude of $E(x)$.

%%%%%%%%%%%%%%
 % Conjecture 3.4
 %%%%%%%%%%%%%% 
\begin{conj} {\em (Montgomery (1987))} 
%Set $E(x) = \Phi(x) - \frac{3}{\pi^2} x^2$.

(1) The remainder term $E(x)$ satisfies as $x \to \infty$ the  bound
$$
E(x) = O ( x \log\log x).
$$

(2) The remainder term $E(x)$ as $x \to \infty$ has maximal order of magnitude given by
$$
E(x) = \Omega_{\pm} (x \log\log x).
$$
\end{conj}

 In 2010  Kaczorowski and Wirtelak \cite{KW10a}, \cite{KW10b}  studied
in more detail  the oscillatory nature  of the remainder term
$E(x)$. These papers show that $E(x)$  can be split  as a sum of  two natural parts, an
arithmetic part and an analytic part, with the
 analytic part having a  direct connection to the zeta zeros.

%************************************************************************
%
%  section 4
%
%
%************************************************************************

\section{Reciprocal Farey product prime power divisibility.}\label{sec4a}

 We now consider the problem of understanding
the  behavior of $\ord_p(\F_n)$.
%************************************************************************
%
%  section 4. 1 Unreduced Farey Products  plots of divisibility
%
%
%**********************************************************************

\subsection{ Farey product prime power divisibility: explicit formula}\label{sec41}

We now turn to prime power divisibility. We obtain the  following direct formula for prime power divisibility of $\ord_p(\F_n)$.
%%%%%%%%
% Theorem 4.1 Farey product prime power
%%%%%%%%
\begin{thm}\label{th43pp}
 The reciprocal Farey product $\F_n$ has prime power divisibility
\[
\ord_p(\F_{n})=  \sum_{b=1}^{\lfloor\log_p(n)\rfloor}\sum_{a=1}^{\big\lfloor\frac{n}{p^b}\big\rfloor}\left(\varphi(ap^b)\left(2- \left\lfloor\frac{n}{ap^b}\right\rfloor\right)- \Big(\sum_{{j|ap}\atop{ d \equiv n\, (\bmod ap^b)}}\mu(j)\left\lfloor\frac{d}{j}\right\rfloor\Big)\right)
\]
where $d\equiv n \, (\bmod \, ap^b)$ with $0 \le d \le ap^b-1$.
\end{thm}

In this formula the M\"{o}bius function appears explicitly, but it is  implicitly present in each Euler totient term $\varphi(a p^b)$ as well.
To prove this result, we write
$F_n = \prod_{r=1}^{\Phi(n)} \rho_r = \frac{N_n}{D_n},$
with $N_n, D_n$ being the product of the numerators (resp. denominators) of all the $\rho_r$.
We  find expressions for $\ord_p(N_n)$ and $\ord_p(D_n)$ separately.

%%%%%%%%%%%%%%
% Lemma Dn-- 4.2
%%%%%%%%%%%%%%
\begin{lem}\label{lem-Dn} The Farey product denominator $D_n$  has
\[
\ord_p(D_n)=\sum_{b=1}^{\lfloor\log_p(n)\rfloor} \Big(\sum_{a=1}^{\big\lfloor\frac{n}{p^b}\big\rfloor}\varphi(ap^b) \, \Big).
\]
\end{lem}

\begin{proof}
The prime $p$ appears in the denominator of a Farey fraction only if the denominator is itself a multiple of $p$. Using this fact, 
the order of $p$ dividing the product of all Farey fractions with denominator is $a_1p^b$, where $1\le a_1\le p-1$, is $b\varphi(a_1p^b)$. 
We count each power of $p$ separately, so to count the $b$-th power we let  $a$ go up to $\lfloor\frac{n}{p^b}\rfloor$, 
this means that a particular $a p^b$ will be counted separately $b$ times. 
This yields the result. \end{proof}

%%%%%%%%%%%%%%%
% Lemma Nn--43
%%%%%%%%%%%%%%%
\begin{lem}\label{lem-Nn} The Farey product numerator $N_n$ has
\begin{equation}
\ord_p(N_n)=\sum_{b=1}^{\lfloor\log_p(n)\rfloor}\sum_{a=1}^{\big\lfloor\frac{n}{p^b}\big\rfloor}\left(\varphi(ap^b)\left(\left\lfloor\frac{n}{ap^b}\right\rfloor-1\right)+\sum_{{j|ap}\atop{ d \equiv n\, (\bmod ap^b)}}\mu(j)\left\lfloor\frac{d}{j}\right\rfloor\right).
\end{equation}
In the last sum $d \equiv  n \,(\bmod \, ap^b)$ with $0\le d<ap^b$.
\end{lem}

\begin{proof}
We count the number of times a given term $ap^b$  appears in the numerators of the Farey fractions,
as the denominators vary from $1$ to $n$.
For any consecutive $a p^b$ denominators, there are $\varphi(ap^b)$ numbers relatively prime to $a p^b$. 
The complete residue system of denominators $ (\bmod \, ap^b)$ is cycled through exactly
$\varphi(ap^b)\left(\left\lfloor\frac{n}{ap^b}\right\rfloor-1\right)$
times. 
Finally there 
is a partial residue system of remaining denominators of Farey fractions
of length $d$, where $d$ is the least nonnegative residue with $d \equiv n \, (\bmod \, ap^b)$.
The relatively prime denominators in this interval are counted
by the  term $ \sum_{j|ap}\mu(j)\left\lfloor\frac{d}{j}\right\rfloor$.
 This concludes the proof.
\end{proof}

%%%%%
% Proof of Theroem 4.1 
%%%%%
\begin{proof}[Proof of Theorem \ref{th43pp}]
Using $\ord_p(\F_n)=\ord_p(D_n)-\ord_p(N_n)$, the result follows from Lemmas \ref{lem-Dn} and \ref{lem-Nn}.
\end{proof}

%%%%%
% Remark 4.4
%%%%%
\begin{rem} \label{rem44}
The value of $\ord_p(\F_n)$ is the result of a race between the contribution of its numerator $\ord_p(D_n)$ and denominator $\ord_p(N_n)$.
These two quantities have a quite different form as arithmetic sums given 
in Lemma \ref{lem-Dn} and \ref{lem-Nn}. In the case of the unreduced
Farey products $\G_n$, the difference between numerator and denominator contributions is very 
pronounced, where the corresponding denominator 
contribution $\ord_p(D_n^{\ast})$ has very large size
at prime powers and is zero when $(n,p) =1$, while the numerator $\ord_p(N_n^{\ast})$ increases at a rather steady rate as a function of $n$.
\end{rem}

The formula of Lemma \ref{lem-Dn} yields the  following estimate of the size of $\ord_p(D_n)$.

%%%%%%%%
% Lemma Dn2-Lemma 4.5-Made changes here (aug17)
%%%%%%%%
\begin{lem}\label{lem-Dn2} For a fixed prime $p$, the  Farey product denominator $D_n$  as $n \to \infty$ has
\[
\ord_p(D_n)= \left(\frac{1}{p-1} - \frac{1}{p^2}\right) \frac{3}{\pi^2} n^2 + O ( n (\log n )^2),
\]
where the implied constant in the $O$-symbol depends on $p$.
\end{lem}

\begin{proof}
We can rewrite the formula of Lemma \ref{lem-Dn} as
$$
\ord_p(D_n) = \sum_{b=1}^{\lfloor \log_p n\rfloor} \Phi\left(\frac{n}{p^b}\right)p^{b-1}(p -1) + \Phi\left(\frac{n}{p^{b+1}}\right) p^{b},
$$
where $\Phi(x)$ is given by \eqref{RH1}.
Here we used the fact  that $\Phi(ap^b) = \Phi(a) p^{b-1}(p-1)$ if $p \nmid a$ and is $\Phi(a) p^{b}$ if $p | a$. We then obtain
$$
\ord_p(D_n) = \Phi\left(\frac{n}{p}\right) (p-1) + \sum_{j=2}^{\lfloor \log_p n\rfloor} \Phi\left( \frac{n}{p^j}\right) p^j.
$$
We use the formula $\Phi(x) = \frac{3}{\pi^2} x^2 + E(x)$  to obtain
$$
\ord_p(D_n) = \frac{3}{\pi^2} \left(\frac{n}{p}\right)^2(p-1) + \sum_{j=2}^{\lfloor \log_p n\rfloor} \frac{3}{\pi^2} \left(\frac{n}{p^j}\right)^2 p^j
 + \tilde{E}(n), 
$$
in which
$$
\tilde{E}(n) := E\left(\frac{n}{p}\right) (p-1) + \sum_{j=2}^{ \lfloor \log_p n\rfloor} E\left( \frac{n}{p^j}\right) p^j.
$$
Using $E(x) = O( x \log x)$ one easily obtains
$$
\tilde{E}(x) = O ( x (\log x)^2).
$$
The main term simplifies to 
$$
\left( \frac{1}{p} + \frac{1}{p^3} +\frac{1}{p^4} +  \cdots + \frac{1}{p^{\lfloor \log_p n\rfloor} } \right) \frac{3}{\pi^2} n^2 = 
\left(\frac{1}{p-1} - \frac{1}{p^2}\right)  \frac{3}{\pi^2} n^2
+ O(n),
$$
as asserted. 
\end{proof}

The quantities  $\ord_p(D_n)$ and $\ord_p(N_n)$ 
must be roughly the same size, because  their difference $\ord_p(\F_n)$ is
of much smaller magnitude. The size of the difference is 
upper bounded using a sharp estimate of the size of 
unreduced Farey products $\G_n$.

%%%%%%%%%%%%%%
 % Theorem 4.6
 %%%%%%%%%%%%%% 
\begin{thm} \label{th44} 
We have
$$
| \ord_p(\F_n)| \le   n (\log n)^2.
$$
\end{thm}
%%%%%%%%%%%%%%
 % Remark 4.7
 %%%%%%%%%%%%%% 
\begin{rem}\label{rem47}
We suggest below that the true order of magnitude of $\ord_p (\F_n)$ is $O(n \log n)$,
see Property (P4)  in  Sect. \ref{sec43}.
\end{rem}

\begin{proof}
Using the  upper bound $\ord_p(\G_n) < n \log_p n$ of Theorem \ref{th29}  together with 
the formula  \eqref{eq42} relating $\ord_p(\F_n)$ to various $\ord_p(\G_n)$ yields
$$
|\ord_p(\F_n)| \le \sum_{k=1}^n |\mu(k)| \ord_p (\G_{\lfloor n/k\rfloor}) \le \sum_{k=1}^n \frac{n}{k} \log \frac{n}{k} \le  n (\log n)H_n  \le n(\log n)^2,
$$
where the last inequality used  the bound $H_n = \sum_{k=1}^n \frac{1}{k} \le \log n$.
\end{proof}

 Theorem \ref{th44} when  combined with Lemma \ref{lem-Dn2} yields the asymptotic estimate
 for the numerators
\begin{equation}
\ord_p(N_n) = \left(\frac{1}{p-1} - \frac{1}{p^2}\right) \frac{3}{\pi^2} n^2 + O ( n (\log n )^2),
\end{equation} 
since $\ord_p(\F_n) = \ord_p(D_n) - \ord_p(N_n)$.

%************************************************************************
%
%  section 4. 2 Farey product prime power divisibility:Empirical dat,a 
%
%
%**********************************************************************

\subsection{Behavior of $\ord_p(\F_n)$: empirical data}\label{sec42}

We made an empirical investigation of 
the prime power divisibility of $\ord_p(\F_n)$
for small primes $p$, and based on the data, we formulate four hypotheses
about the behavior of these functions. 
The amount of the computation increases as $p$ increases,
and we present data here for  $p=2$, and for $p=3$ in an Appendix). 
Figure \ref{fig41-ord2} plots  the values of $\ord_2(\F_n)$, ordered by $n$.

The distribution of points for $\ord_2 (\F_n)$ is more scattered than for $\ord_2(\G_n)$ 
(compare Figure \ref{fig21-ord2}) and  includes
many  negative values.  The ``streaks"  in $\ord_2(\G_n)$ visible in Figure \ref{fig21-ord2} are gone.
Figure \ref{fig41-ord2}  shows large positive  jumps in $\ord_2(\F_n)$ between $n= p^k -1$ and $n=p^k$ for $p=2$.
This fact can be proved for all primes $p$, by noting that
$$
\ord_p(\F_{p^k}) - \ord_p(\F_{p^k -1}) = k p^{k-1}(p-1).
$$
 This jumping behavior at powers of $p$ parallels  that for $\ord_{p} (\G_n)$,
 where \eqref{sharp-bd} states
  $$
 \ord_p(\G_{p^k}) - \ord_p(\G_{p^k -1}) = k p^k - \frac{p^k -1}{p-1}.
 $$
We see that the jump magnitude for $\ord_p(\F_{p^k})$ is scaled down from that of $\ord_p(\G_{p^k})$
by a factor approximately $1-1/p$.

%%%%%%%%%%%%%%
%Figure 4.1 - ord_2 (F_n) plot
%%%%%%%%%%%%%%%

\begin{figure}[!htb]
\includegraphics[width=130mm]{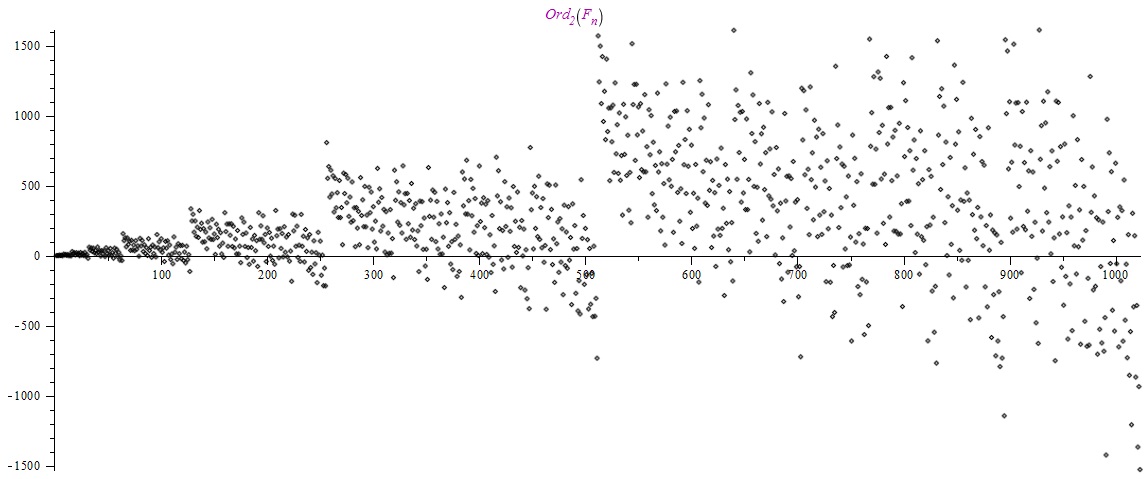}
%ord2fnlinear.jpg}
\caption{Values of $\ord_2(\F_n)$, $1 \le n \le 1023$.}
\label{fig41-ord2}
\end{figure}

We next consider the empirical distribution of the individual values of $\ord_2(\F_n)$.
Figure \ref{fig42-ord2hist} plots  a the rescaled values  $\frac{1}{n} \ord_2(\F_n)$ on
the interval between $2^k \le n < 2^{k+1}$, ordered by size.

%%%%%%%%%%%%%%
%Figure 4.2
%%%%%%%%%%%%%%%

\begin{figure}[!htb]\includegraphics[width=130mm]{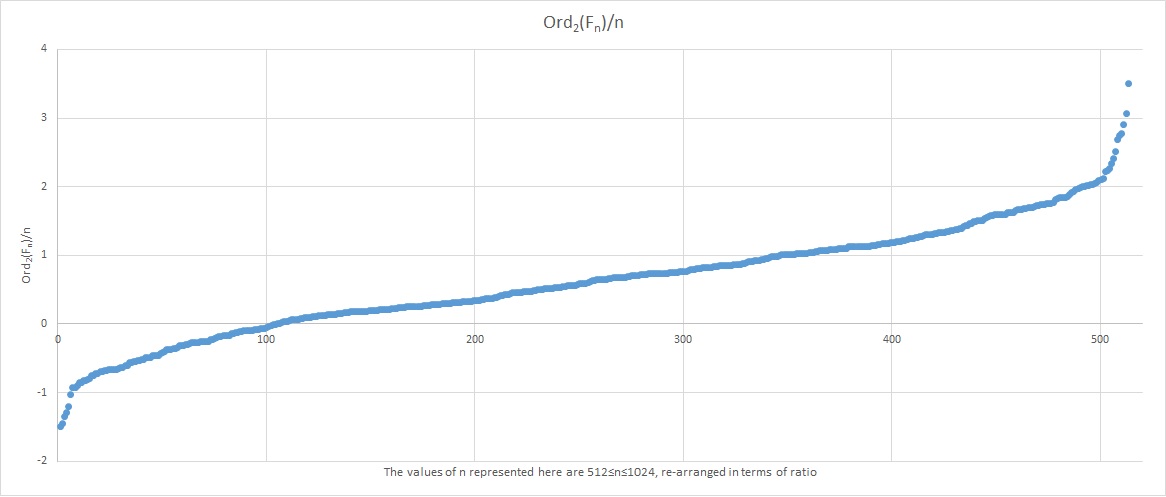}

\caption{Sorted values  of $\frac{1}{n}\ord_2(\F_n)$, $512 \le n \le 1024$.}
\label{fig42-ord2hist}
\end{figure}

This plot looks qualitatively similar to that for $\ord_2(\G_n)$  in Figure \ref{fig32-CLT},
with the change that the median of the distribution is shifted downwards. 
The  median of this empirical distribution is  around $0.7$, suggesting that the
average value of $\ord_2(\F_n)$ is around $0.7 n$ on this range $512 \le n \le 1024$.
In particular the median appears to be  much smaller than $\frac{1}{2} n \log_2 n$ for $\ord_2(\G_n)$.
The data is insufficient to guess at what rate the median of the distribution is
growing: is it growing like $Cn$ or like $C n \log_2 n$?

Finally we study jumps of the function at $n=p^r -1$. 
Empirical data suggests  that $\ord_p(F_{p^r -1})$ may be always non-positive, as shown for $p=2$  in Table 4.1 below.
The last two columns suggest that these values seem to grow like a constant times $n \log_2 n$.
In Appendix A we present additional data  for $p=3$, for $1 \le r \le 10$, where we observe similar behavior occurs. \\
 
%***********************************************************************
%
% TABLE 4.1 --
%
%************************************************************************

\begin{minipage}{\linewidth}
\begin{center}
\begin{tabular}{|r | r | r | c |c|}
\hline
%Set &
 \mbox{Power $r$} & $N= 2^r -1$ & $\ord_2(\F_{2^r -1})$  &  $-\frac{1}{N} \ord_2(F_{2^r -1})$  & $-\frac{1}{N \log_2 N} \ord_2(F_{2^r -1})$ \\
%\multicolumn{2}{c}{$\alpha_k$}{c{Hausdorff dim}  \\
\hline
 $1$ &   $1$& $0$ &  $0.0000$  &  $0.0000$\\
$2$ &   $3$ & $0$ & $ 0.0000$  &  $0.0000$\\
$3$ &  $7$ & $-1$  &   $0.1429$  & $0.0509$\\ 
$4$ & $15$ &  $-2$  & $0.1333$   &   $0.0341$\\ 
$5$ &  $31$ & $-19$  & $0.6129$  &  $ 0.0586$\\ 
$6$ & $63$ &   $-35$    & $0.5555$  & $0. 0929$\\
$7$ & $127$ & $-113$ & $0.8898$  & $0.1273$\\
$8$ &  $255$ & $-216$ & $0.8471$  & $0.1095$\\ 
$9$ &  $511$ & $-733$ & $1.4344$  & $ 0.1594$\\ 
$10$ & $1023$ & $-1529$ & $1.4946$ & $0.1495$ \\ 
 $11$ &  $2047$ &  $-3830$ &  $1.8710$ &  $0.1701$\\
$12$ &  $4095$ &  $-7352$ & $ 1.7953$ &  $ 0.1496$ \\
$13$ &  $ 8191$ & $-20348$  &   $2.4842$ &  $0.1910$\\ 
$14$ &  $16383$ & $-41750$  & $2.5484$  & $0.1820$ \\ 
$15$ &  $ 32767$ & $-89956$  & $2.7453$  &  $0.1830$ \\ \hline
\end{tabular} \par
\bigskip
\hskip 0.5in {\rm TABLE 4.1.}  
{\em Values at $N= 2^r -1$ of $\ord_2(\F_N)$.}
%\label{Table41}
%and $p=3$
\newline
\newline
\end{center}
\end{minipage}

%************************************************************************
%
%  section 4. 3 Farey product prime power divisibility:hyptothetical properties%
%
%**********************************************************************

\subsection{Behavior of $\ord_p(\F_n)$: hypothetical properties}\label{sec43}

The empirical data in Figures \ref{fig41-ord2}  and \ref{fig42-ord2hist} 
together with Table 4.1 suggest that  
 the following  (unproved) hypothetical properties  (P1)-(P4) might conceivably hold for all  the functions $f_p(n) := \ord_p(\F_n)$
The first property concerns the sign of $\ord_p(\F_n)$ at $n = p^k-1$.\medskip

%%%%%%%

\noindent {\bf Property (P1).}
{\em For a given prime $p$ there holds
$$
f_p({p^k -1}) \le 0 \quad\mbox{ for all}\quad  k \ge 1. 
$$
Furthermore $f_p(p^k -1) < 0$ 
for all $k \ge 2$, with the exception $(p,k)=(2,2).$}.\medskip

The second property concerns the sign of $\ord_p(\F_n)$ at $n=p^k$.\medskip

\noindent {\bf Property (P2).}
{\em For a given prime $p$ one has $f_p({p^k})>0$  for all   all $k \ge 1$.}.\medskip

The third property concerns sign changes of $\ord_p(\F_n)$. \medskip

\noindent{\bf Property (P3).}
{\em For a given prime $p$ the inequalities $f_p(n) >0$ and $f_p(n) <0$ each occur infinitely often.
Each may  hold for a positive proportion of $n$, as $n \to \infty$.}\medskip

The fourth property concerns the absolute magnitude of $|\ord_p(\F_n)|$.\medskip

\noindent {\bf Property  (P4).}
{\em For a given prime $p$ there are are finite positive constants $C_{1, p}, C_{2, p}$ such that, for all $n \ge 1$,}
$$
- C_{1, p} \, n \log_p n \le f_p(n) \le C_{2,p} \, n \log_p n.
$$

We are  far from establishing the validity of any of  Properties (P1)-(P4) for $f_p(n) = \ord_p( \F_n)$.
Because the  fluctuations in M\"{o}bius function sums remain small for $n \le 10000$, the computational
evidence presented is a rather limited test of these properties.
We are not completely  convinced they are true.
Perhaps Property (P1) holds for a given $p$ only for 
 $k$ sufficiently large.  In the next subsection we present 
 limited  theoretical  evidence in their favor.

%************************************************************************
%
%  section 4.4
%
%
%************************************************************************

\subsection{Evidence for  hypothetical properties (P1)-(P4). }\label{sec44}

Properties (P1) and (P2) hold for the $k =1$ case of $\ord_p(\F_n)$. 
We  have verified computationally
 that  Properties (P1), (P2) hold for all primes $p <1000$ when  $k=2$. 
We have verified that Hypotheses (P1), (P2) hold 
for $p=2$ for exponents $1 \le k \le 15$ and that for  $p=3$ for exponents $1\le k \le 10$. 

An interesting  special case  to consider  is   whether
$\ord_p (\F_{p^2 -1}) < 0$ holds for all $p \ge 3$.  Note that this function of $p$ is
complicated because it involves all values $\{ \mu(k) : 1 \le k \le p^2\}.$
To aid in its study, we give  several formulas for this function.

%%%%%%%%%%%%%%
 % Theorem 4.8
 %%%%%%%%%%%%%% 
\begin{thm}\label{th46a}
Let  $p \ge 3$ be prime.

(1) One has
\begin{equation}\label{p21}
\ord_p( \F_{p^2 -1}) = \sum_{k=1}^{p-1} \mu(k) \, \ord_p(\G_{ \lfloor \frac{p^2-1}{k} \rfloor})
\end{equation}

(2) For $1 \le k \le p-1$ write
$\lfloor \frac{p^2-1}{k} \rfloor= a_k p + b_k, ~~~ 0 \le a_k, b_k \le p-1,$
then
\begin{equation}\label{closedfm}
\ord_p( \G_{\lfloor \frac{p^2-1}{k} \rfloor})= a_k(p-1- b_k).
\end{equation}
Here $a_k = \lfloor \frac{p-1}{k} \rfloor$
and $b_k = \lfloor \frac{p^2-1}{k} \rfloor - p \lfloor \frac{p-1}{k} \rfloor.$

(3) One has 
\begin{equation}\label{p21-alt}
\ord_p(\F_{p^2 -1}) = (p-1) - \Big(\sum_{k=1}^{p-1} \mu(k) \lfloor \frac{p-1}{k} \rfloor b_k\Big).
%%%%%%%
% - \frac{ (p-1)^2}{4}+ (p-1) (\sum_{k=3}^{p-1} \mu(k) \lfloor \frac{p-1}{k} \rfloor) - \sum_{k=3}^n  \mu(k) b_k
%%%%%%%
\end{equation}
\end{thm}

%%%%%%%%%%%%%%
 % Remark 4.9
 %%%%%%%%%%%%%% 
\begin{rem}
In particular  whenever $k| (p-1)$ one has  $a_k=b_k = \frac{p-1}{k}$; 
these values include  $k=1, 2, \frac{p-1}{2}, p-1$. One has $\ord_p(\G_{p^2-1}) =0$ and $\ord_p( \G_{(p^2-1)/2}) = \frac{(p-1)^2}{4},$
and the $k=2$ term makes a large negative contribution.  This fact is
sufficient to  explain the negativity of $\ord_p(\F_{p^2-1})$ for small primes.
\end{rem}
\begin{proof}
(1) We have $\ord_p(\G_n)=0$ for $1 \le n \le p-1$, and $\lfloor \frac{p^2-1}{j} \rfloor \ne p$
for all integers $j$. The M\"{o}bius inversion formula \eqref{eq41} has all terms vanish for $\ell \ge p$,
which yields \eqref{p21}.

(2) For $1 \le k \le p-1$
we have 
$$
a_k= \lfloor \frac{p^2 -1}{pk} \rfloor = \lfloor \frac{p-1}{k} \rfloor,
$$
whence
$$
b_k = \lfloor \frac{p^2 -1}{k} \rfloor -p\lfloor \frac{p-1}{k} \rfloor
$$

Using Theorem \ref{th39}
we have 
$$
\ord_p(\G_n) = \frac{1}{p-1}( 2\Sum_p(n) - (n-1) d_p(n)).
% = a_k( p-1-b_k).
$$
Substituting $n= a_k p + b_k$ we find that
$$2\Sum_p(n) = (a_k+ b_k)^2 -(a_k + b_k) + a_k (p -1)(a_k+ p-1),$$
while
$$(n-1) d_p(n) = (a_k^2 + a_k b_k- 2a_k) p + (b_k-1)(a_k+b_k)+a_k.$$
A calculation yields \eqref{closedfm}.

(3) We have  the identity, valid for all $n \ge 1$, 
\begin{equation}\label{magic}
\sum_{k=1}^n \mu(k) \lfloor \frac{n}{k} \rfloor = 1,
\end{equation}
It is easily proved by induction on $n \ge 1$.
Substituting the formula of (2) into that of (1) gives
the result.
\end{proof}

%%%%%%%%%%%%%
% Figure 4.3
%%%%%%%%%%%%%%
% Ord_3 F_n
%%%%%%%%%%%%%%%

\begin{figure}[!htb]
\includegraphics[width=130mm] {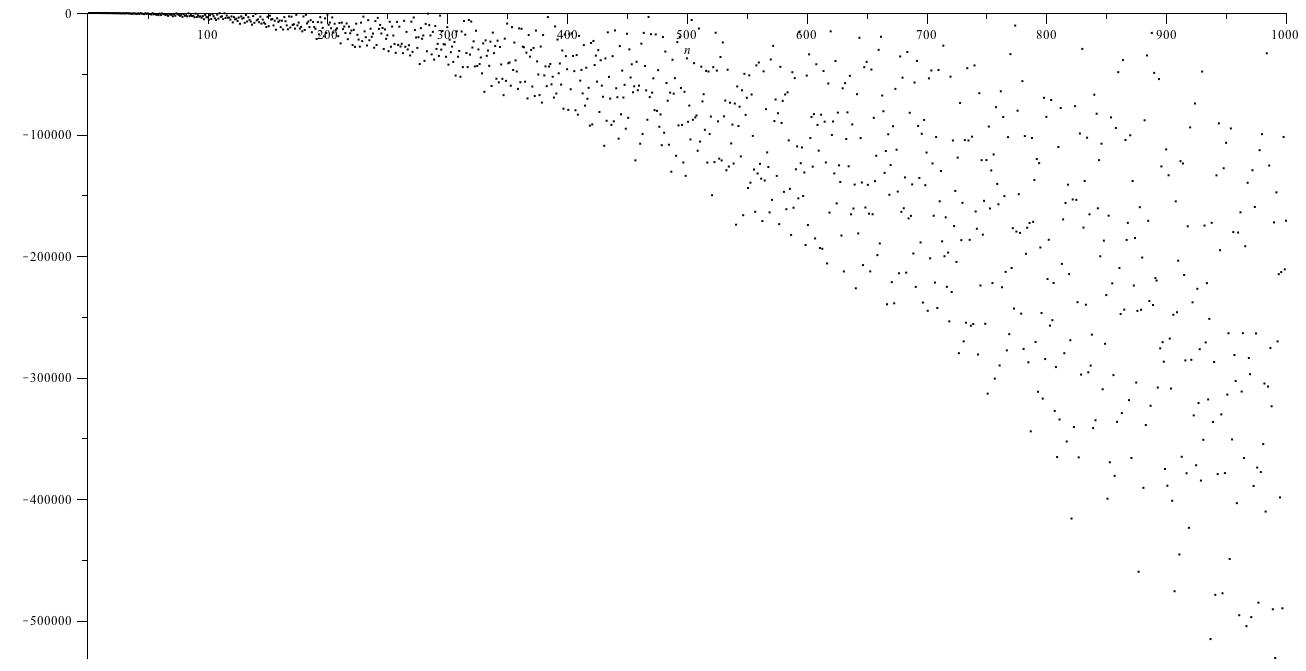}
\caption{Values of $\ord_p(\F_{p^2-1})$ for  prime $1 \le p \le 1000$.}
\label{fig43}
\end{figure}

Figure \ref{fig43}   plots   $\ord_p( F_{p^2 -1})$ for $3 \le p \le 1000$.
 The distribution of these values 
has a lower envelope which appears  empirically\footnote{
The data in  Figure \ref{fig43} seems insufficient to discriminate between growth of order $N$ and of order $N \log_p N$.
For $N=p^2 -1$ the quantity $\log_p N \approx 2$ is approximately constant. }
to be of the form
$-c N \log_p N$ with  $c=0.25$, where  $N= p^2 -1.$ 
  It  has a pronounced scatter of points including some
 values rather close to $0$, but never crossing $0$. 
 The observation suggests that there may be a barrier at $0$,
 and one may ask:  {\em Is there some arithmetic interpretation of the values $\ord_p( F_{p^2 -1})$ that
might justify their negativity, i.e. the truth of Property (P1) for $k=2$?}

As an initial step in the direction of  Property (P3), we show 
 that $\ord_p(\F_n)$ takes positive and negative values at least once,
 for each prime $p$.

%%%%%%%%%%%%%%
 % Theorem 4.10
 %%%%%%%%%%%%%% 
\begin{thm}\label{th45}
 For each prime $p$ the function $\ord_p(\F_n)$ takes both
positive and negative values. 
\begin{enumerate}\item
For each $p \ge 2,$  $\ord_p(\F_{p}) >0$ with 
$\ord_p(\F_{p}) = p-1.$
\item
For $p=2$,  $\ord_2(\F_{7}) = -1$.
For odd primes  $p$, 
$$\ord_p(\F_{3p-1}) = -(\frac{p-1}{2}),$$
More generally, for $p \ge 3$, 
\begin{equation}\label{eq51a}
\ord_p(\F_n) <0 \,\,\, \mbox{for} ~~~ \frac{8}{3}p \le n \le 3p-1.
\end{equation}
\end{enumerate}
\end{thm}

\begin{proof} 
Write
$F_n = \frac{N_n}{D_n}$ where $N_n$ is the product of the numerators of
the positive Farey fractions $\frac{h}{k}$ of order $n$, and $D_n$ is the product of the denominators.
(The  quantities  $N_n$ and $D_n$ will  have a large common factor.) Now the reciprocal Farey product has 
$$
\ord_p(\F_n) = \ord_p(D_n) - \ord_p(N_n).
$$
 Choosing $n=p$, $\ord_p(N_p)=0$
while $\ord_p(D_p) = \varphi ( p) = p-1$. 

To find negative values,  calculation gives  $\ord_2(\F_{7}) =-1$.
Suppose now $ p \ge 3$.
For $2p \le n \le 3p-1$ we have $\ord_p(D_n) = 2(p-1)$,
coming from the denominators $p$ and $2p$. For $p+1\le n \le 2p-1$
the Farey fraction $\frac{p}{n}$ contributes to $\ord_p(N_m)$, for
any $m \ge n$. For 
$2p+1 \le n \le 3p-1$ the fraction $\frac{p}{n}$ similarly contributes one to  $\ord_p(N_m)$, as does $\frac{2p}{n}$ for odd values of $n$ in this interval.
We conclude that for $2p+1 \le n \le 3p-1$,
$$
\ord_p(N_n) = (p-1) + \lfloor \frac{3}{2}(n- 2p) \rfloor
$$
This yields $\ord_p(N_n) \ge 2p -1$ for $n \ge \frac{8}{3}p$,
whence $\ord_p(\F_n) <0$, giving (2).  Finally, choosing $n=3p-1$
we obtain $\ord_p(\F_{3p-1}) = - (\frac{p-1}{2}).$ 
\end{proof}

In the direction of  Property (P4), we have the  weak bound
$$|\ord_p(\F_n)| = O ( n (\log n)^2)$$
 given in Theorem \ref{th44} above. We also have the Omega result
 $$\ord_p (\F_n) = \Omega( n \log_p n),$$
 because  the  individual jumps in the function $\ord_p(\F_n)$  are at least as large as a constant times
 $ n \log_p n$. Indeed, for  $n=p^k$ we have 
$$
\ord_p(\F_{p^k}) - \ord_p( \F_{p^k -1}) = k \varphi(p^k) = k p^{k-1}(p-1) = (1- \frac{1}{p}) n \log_p n.
$$
This calculation implies that 
$$
\limsup_{n \to \infty} \frac{ \ord_p(\F_n)}{n \log_p n} - \liminf_{n \to \infty}  \frac{ \ord_p(\F_n)}{n \log_p n} \ge 1- \frac{1}{p}.
$$
Thus  the assertion of Property (P4), if true, is qualitatively best possible. 

%%%%%%%%%%%%%%%%%%%%%%%%%%%%%%%%%%%%%%%%%%%%%%%%%%%%%%%%%%%%%%
%Question: One can also ask, does  the equality hold in this limit, 
%and, even more, does the lim inf take value $-\frac{1}{2}(1- \frac{1}{p})$ and the lim sup  value $\frac{1}{2}(1+ \frac{1}{p})$?
%%%%%%%%%%%%%%%%%%%%%%%%%%%%%%%%%%%%%%%%%%%%%%%%%%%%%%%%%%%%%%

%************************************************************************
%
%  section 4.5
%
%
%************************************************************************

\subsection{When is the reciprocal Farey product $\F_n$ an integer?}\label{sec45}

 This question was originally raised (and solved) in \cite{DL11b}.
Their  solution was obtained using  \eqref{eq51a} in  Theorem \ref{th45}, as follows.
%%%%%%%%%%%%%
%Theorem 4.11
%%%%%%%%%%%%
\begin{thm}\label{th46}
Finitely many  reciprocal Farey products $\F_n$ are integers.
The largest such value is $n=58$.
\end{thm}

\begin{proof}

If $n$ has the property that there exists a prime $p$ satisfying
$$
\frac{1}{3}(n+1) \le p \le \frac{3}{8}n,
$$
then condition \eqref{eq51a} of Theorem \ref{th45} will be satisfied and 
 $\ord_p(\F_n) <0$ certifies that $F_n$ is not an integer. 
The prime number theorem implies that for any $\epsilon >0$ and all sufficiently large $n$
 the interval $(\frac{1}{3}n, \frac{3}{8}n]$
contains at least  $\frac{1}{24} (1- \epsilon) \frac{n}{\log n}$ primes.
In particular  such  a prime will exist for all sufficiently large $n$, whence there are only finitely
many integer $\F_n$.
 
 To obtain the  numerical bound $n=58$ requires the use of prime counting estimates with explicit remainder
 terms, together with computer calculation for small $n$, described in  the solution cited in \cite{DL11b}. 
\end{proof}

%************************************************************************
%
%  section 4.5
%
%
%************************************************************************

\subsection{Reciprocal Farey product $\F_n$ given  in lowest terms}\label{sec45}

Now consider the reciprocal Farey product $\F_n$ as a rational fraction given
in lowest terms, calling it $\F_n= \frac{\hat{D}_n}{\hat{N}_n}$, with
$$
\hat{D}_n := \frac{D_n}{(N_n,D_n)}\quad \mbox{and} \quad \hat{N}_n := \frac{N_n}{(N_n, D_n)}.
$$
We  ask: {\em What are the growth rates of $\hat{D}_n$ and $\hat{N}_n$?}

We have no  answer to this question and about it make the following remarks.
\begin{enumerate}
\item[(i)] It is not  clear whether $\log(\hat{N}_n)$ and $\log(\hat{D}_n)$
separately have  smooth asymptotic behaviors as $n \to \infty$. However   their difference does,
since
$$
\log( \F_n) = \log( \hat{D}_n) - \log( \hat{N}_n) = \frac{3}{\pi^2} n^2+O( n \log n),
$$
as follows using Theorem \ref{th43}9(1),   \eqref{phi-bound}  and the estimate $\psi(n) = O(n)$.
\item[(ii)] The function $\hat{N}_n$ initially grows much 
 more slowly than $\hat{D}_n$.
Theorem \ref{th46} gives $\hat{N}_{58} =1$, while $\hat{D}_{58} > 10^{400}$.
However Theorem \ref{th45}(2) implies a nontrivial asymptotic lower bound for  
growth of $\log(\hat{N}_n)$. It states 
that the product of all primes in the range $\frac{1}{3} n < p < \frac{3}{8} n$ divides $\hat{N}_n$,
which since there are $\gg \frac{n}{\log n}$ prime numbers in this interval implies 
that there is a positive constant $c$ such that 
$\log(\hat{N}_n) \gg  n$
 for  all sufficiently large $n$.

\item[(iii)] We do not know what is the maximal order of growth of $\log(\hat{N}_n)$. Properties (P3) and (P4),
if true,  allow the possibility that it could be close to the same order as the main term. That is,  they suggest the possibility that
there is a positive constant $c$ such that $\log(\hat{N}_n) > c n^2$ infinitely often.  
\end{enumerate}

%%%%%
% section 5
%%%%%
%************************************************************************
%
%  section 5
%
%
%************************************************************************

\section{Farey product archimedean  encoding  of the Riemann hypothesis}\label{sec5}

We have already seen that results of 
 Mikol\'{a}s encode the Riemann hypothesis in terms of $\F_n$ via a formula 
 $$
 \W_{\infty}(\F_n) := \log(\F_n) = \left(\Phi(n) - \frac{1}{2} \psi(n)\right) + R_{\F}(n),
 $$
 which has the arithmetic main term $\Phi(n)-\frac{1}{2}\psi(n)$ on the right side, plus a
 remainder term $R_{\F}(n)$.
 The equivalence to the  Riemann hypothesis is formulated as the remainder term bound $R_{\F}(n)= O( n^{1/2 + \epsilon})$.
The  arithmetic main term  has the feature  that it has oscillations in lower-order terms of its asymptotics which  
 are of size much bigger than the remainder term $R_{\F}(n)$; thus, this arithmetic main term is a complicated object,
 whose behavior is of interest in its own right.

In this section we will show that one can  replace the arithmetic main term of Mikol\'{a}s 
on the right side of his formula with a new arithmetic main term $\Phi_{\infty}(n)$
built entirely out of the quantities $\log(\G_{k})$ associated to unreduced Farey products $\G_k$.
To do this we make use of 
 the M\"{o}bius inversion formula in Theorem \ref{th41}, and the splitting in \eqref{archimedean}
 \footnote{Here $\Phi_{\infty}(n)$
 is the ``replacement main term" mentioned in Sect. \ref{sec12} and defined in  
 \eqref{arch-mainterm} below.}.
 The advantage of our reformulation is that with it one can define formal analogues   for each finite prime $p$.
On the left side, the quantity to approximate,
$\log(\F_n)$,  has an analogue quantity  defined for each prime,  $\ord_p(\F_n)$.
On the right side, the  new arithmetic main term $\Phi_{\infty}(n)$ we introduce  has analogue
quantities  built out
of replacing the quantities $\log(\G_k)$ with $\ord_p(\G_k)$ 
in suitable ways. 
This permits us to attempt reformulations of the Riemann hypothesis at each prime $p$ separately,
as we describe in Section \ref{sec6}.
%%%%%%%%%%%%%%%%%%%%%%%%%%%%%%%%%%%%%%%%%%%%%%%%%%%%%
%That is, we propose new   arithmetic main terms both at the archimedean place and at all
%the prime places, which necessarily have oscillations,
%for which we hope  the Riemann hypothesis may be encoded in the size of the remainder term.
%These arithmetic main terms  are constructed using  unreduced Farey products $\G_n$.
%%%%%%%%%%%%%%%%%%%%%%%%%%%%%%%%%%%%%%%%%%%%%%%%%%%%%%

%************************************************************************
%
%  section 5.1
%
%
%************************************************************************

\subsection{Farey product archimedean arithmetic main term}\label{sec51}

We introduce our new archimedean arithmetic term $\Phi_{\infty}(n)$ at the real place,
and its associated remainder term $R_{\infty}(n)$ defined by

\begin{equation}\label{arch2}
\R_{\infty}(n):= \log (\F_n) -\Phi_{\infty}(n)   
\end{equation}
The   {\em archimedean  arithmetic term} $\Phi_{\infty}(n)$ is given by
\begin{equation}\label{arch-mainterm}
\Phi_{\infty}(n) :=   \sum_{k=1+\Kn}^n \mu(k) \Phi^{*}\Big( \left\lfloor \frac{n}{k} \right\rfloor\Big)+ 
 \sum_{k=1}^{\Kn} \mu(k)\, \log (\G_{\lfloor n/k \rfloor}), 
\end{equation}

\noindent  in which the  function $\Phi^{*}(n) = n(n+1)/2$ counts the number of unreduced Farey products
 of order $n$, and we choose  a cutoff $K_n \approx \sqrt{n}$. 
 By collecting all terms with $\lfloor n/k\rfloor =\ell$ 
we may rewrite the archimedean arithmetic term above in the alternate form 

\begin{equation}\label{arch2a}
\Phi_{\infty}(n) =   \sum_{k=1}^{\Kn}  \mu(k)\, \log (\G_{\lfloor n/k \rfloor}) +
\sum_{\ell =1}^{\Ln}  \Big( M\left(\frac{n}{\ell}\right) - M\left(\frac{n}{\ell+1}\right) \Big)  \frac{\ell(\ell+1)}{2},
\end{equation}

\noindent in which $\Ln \approx \sqrt{n}$ is determined by $\Kn$, and vice versa.
Using \eqref{eq41} and \eqref{archimedean} we can express the remainder term $\R_{\infty}(n)$
as 
\begin{equation}
\R_{\infty}(n) = \sum_{\ell = 1}^{ \Ln}  \left( M\left(\frac{n}{\ell}\right) - M\left( \frac{n}{\ell+1} \right)\right) \Big( \log (\G_{\ell})- \frac{\ell(\ell+1)}{2} \Big).
\end{equation}
 For calculations reported below we chose
%%%%%%%%%%%%%%%%%%%%%%%%%%%%%%%%%%%%%%%%%%%%%%%%%%%%%%%%%%%%%%%%%
%\footnote{There are alternate choices for a cutoff $\Kn$ and $\Ln$.which effect the distribution
% of a single term in the sum, which is of size $O (\sqrt{n}\log n)$
%and so  can be absorbed in the error terms in all later arguments.
%An optimal cutoff that preserves a certain symmetry of the two terms  is discussed in \cite{AL14}.
%For $m^2 \le n < m(m+1)$ it splits the term 
%$\mu(m)(\G_{\lfloor n/m \rfloor}- \frac{1}{2} m (m+1))$ and assigns half of it to each of  the ``main term" and the ``remainder term." }
%%%%%%%%%%%%%%%%%%%%%%%%%%%%%%%%%%%%%%%%%%%%%%%%%%%%%%%%%%%%%%%%%%%%
\begin{equation}\label{elln}
\Ln = \lfloor \sqrt{n} \rfloor =m,
\end{equation}
in which case we have
\begin{equation}\label{Knn}
\Kn =\begin{cases}
\lfloor \sqrt{n} \rfloor -1 & \quad \mbox{for} \quad m^2 \le n < m(m+1),\\
\lfloor \sqrt{n} \rfloor & \quad \mbox{for} \quad m(m+1) \le n < (m+1)^2.
\end{cases}
\end{equation}
The  definition \eqref{arch-mainterm} of the archimedean arithmetic term  includes an initial sum
that extends over {\em the full range of summation $1 \le k \le n$.} 
This term is the   contribution under M\"{o}bius inversion
of the main term $\frac{1}{2} (\lfloor n/k \rfloor)( \lfloor n/k \rfloor +1)$ in the asymptotic formula for $\G_{\lfloor n/k \rfloor}$.
%%%%%%%%%%%%%%%%%%%%%%%%%%%%%%%%%%%%%%%%%%%%%%%%%%%%%%%%%%%
%given by Theorem \ref{th20} over the {\em full range of summation}. 
%%%%%%%%%%%%%%%%%%%%%%%%%%%%%%%%%%%%%%%%%%%%%%%%%%%%%%%%%%%%
The second sum in our archimedean main
has  the summation range {\em from $1$ up to about $\sqrt{n}$}. 
It is a ``main term" obtained when using the Dirichlet hyperbola method for splitting sums 
$$
\sum_{k=1}^n F(\frac{n}{k}) g(k) =\sum_{1 \le k \le \Kn}  F(\frac{n}{k}) g(k) + \sum_{\Kn < k \le n}  F(\frac{n}{k}) g(k).
$$
into a ``main term" and ``remainder term", compare
\cite[Lemma 2.1]{Dia82}, \cite[Sect. 3.4]{Lag13}. 
The two terms in the definition of $\Phi_{\infty}(n)$ account for
the two parts of the Mikol\'{a}s arithmetic main term, as explained below. 

A  justification for our definition of $\Phi_{\infty}(n)$ is  the following result.

%%%%%%%
% Theorem 5.1
%%%%%%%

\begin{thm}\label{th51}
The Riemann hypothesis implies that for fixed $\epsilon >0$ as $n \to \infty$
\begin{equation} \label{51main}
\R_{\infty}(n) = O( n^{\frac{3}{4} + \epsilon}),
\end{equation}
where the implied $O$-constant depends on $\epsilon$.
\end{thm}

We defer the proof of Theorem \ref{th51}  to Section \ref{sec53}.
The proof shows  that the initial sum  on the
right side of \eqref{arch-mainterm} is unconditionally of size $\Phi(x) + O(1)$ and  shows
 that  the second sum on the right side of \eqref{arch-mainterm}
 is, conditional on the Riemann Hypothesis,  of size  $-\frac{1}{2} \psi(x) + O ( x^{\frac{3}{4} + \epsilon})$.

Based on Theorem \ref{th51} we  propose: \medskip

\noindent {\bf Hypothesis $\R_{\infty}$.} {\em For each $\epsilon >0$ there holds, as $n \to \infty$,
\begin{equation} 
\R_{\infty}(n) = O( n^{\frac{3}{4} + \epsilon}),
\end{equation}
where the implied $O$-constant depends on $\epsilon$.}

Theorem \ref{th51}  seems weaker in appearance than  the result of Mikol\'{a}s in having a remainder term bounded by $O(x^{3/4+ \epsilon})$
rather than $O( x^{1/2 +\epsilon}),$ so it may seem that Hypothesis $\R_{\infty}$ might be weaker
than the Riemann hypothesis. 
 Subsequent work   of the first author with R. C.Vaughan  will show   that the
converse of Theorem \ref{th51} holds, and  that Hypothesis $\R_{\infty}$ is actually equivalent to
the Riemann hypothesis,
In addition it will show  the true magnitude of the error term is $\Omega(x^{\frac{3}{4}- \epsilon})$.

%%%%%%%%%%%%%%%%%%%%%%%%%%%%%%%%%%%%%%%%%%%%%%%%%%%%%%%%
%{Subsequent work of the first author with Vaughan (\cite{LV14}) 
%expects to establish the converse of Theorem \ref{th51},  showing that Hypothesis $\R_{\infty}$  is equivalent to the
%Riemann hypothesis.  }
% to Theorem \ref{th51} holds, so a priori  the
% Hypothesis $\R_{\infty}$ makes a weaker assertion than the Riemann hypothesis.
%%%%%%%%%%%%%%%%%%%%%%%%%%%%%%%%%%%%%%%%%%%%%%%%%%%%%%%%

 %%%%%%%%%%%%%%%%%%%%%%%%%%%%
%On the other hand, we do not know whether there might be additional cancellation in this remainder term,
%not explained by the Riemann hypothesis,  and ask:
%Might the size of $\R_{\infty}(n)$
%satisfy a stronger bound $\R_{\infty}(n) = O (n^{\theta + \epsilon})$ for some $\theta < \frac{3}{4}$,
%especially $\theta = \frac{1}{2}$?  
%We note that by Theorem \ref{thA1}  the first term on the right is  
%their arithmetic main term in Kaczorowski and Wiertelak \cite{KW10a}, and the second term 
%therefore must coincide with their  analytic term related to zeta zeros.
%%%%%%%%%%%%%%%%%%%%%%%%%%%%%%%%

%************************************************************************
%
%  section 5.2
%
%
%************************************************************************

\subsection{Remainder term $R_{\infty}(n)$: experimental data} \label{sec52}
Figure \ref{fig51-R-inf} presents empirical data on $\R_{\infty}(n)$. The function is bounded by
$n^{3/4}$ over the given range, and
its graph  has a striking appearance exhibiting  definite internal structure. 
%%%%%%%%%%%%%%
%Figure 5.1 - ord_2 (F_n) plot
%%%%%%%%%%%%%%%

%%%%%
% PAGE BREAK
%%%%%5
%\newpage
%%%%

\begin{figure}[!htb]
\includegraphics[width=125mm]{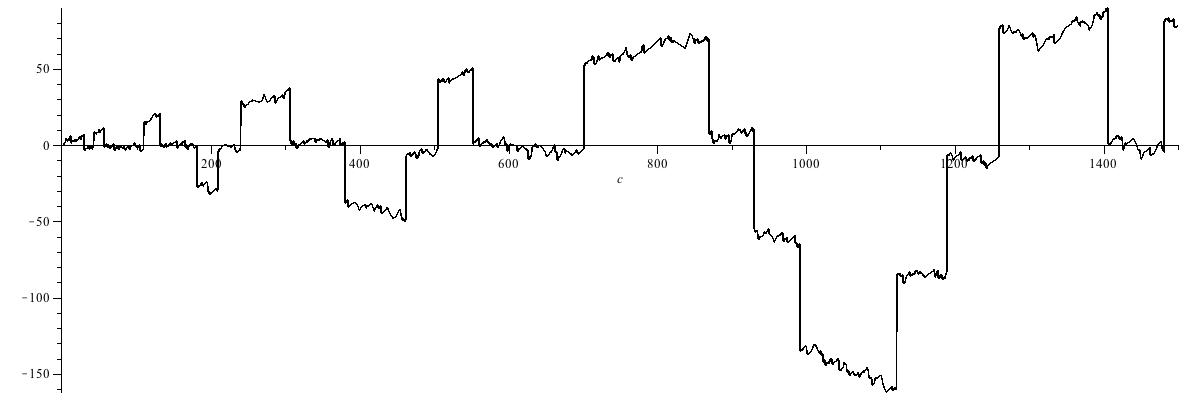}
%ord2fnlinear.jpg}
\caption{$\R_{\infty}(n)$, $1 \le n \le 1500$.}
\label{fig51-R-inf}
\end{figure}

The graph exhibits occasional  large jumps of varying sign followed by slow variation of the function. 
It was noted by J. Arias de Reyna  
that the location of these  jumps of the function visible in the graph in Theorem \ref{fig51-R-inf} 
are at a subset of the points $n=m(m+1)$.
Subsequent work related these jumps to the hyperbola method splitting  of
the ``main term" and ``remainder term". 
They  occur only at values  $m$ is squarefree,
and the direction of each jump is that of   $\mu(m)$.

%%%%%%%%%
% Subsequent work of the first author with R. Vaughan  shows that the converse holds.
%To show this requires  showing that if the Riemann hypothesis fails, the second term on the right side of \eqref{arch-mainterm} 
%is $\Omega( n^{\frac{1}{2} + \delta}) $ for some fixed $\delta >0$.
%\begin{proof}
%%%%%%%%%%

%************************************************************************
%
%  section 5.3 
%
%
%************************************************************************

\subsection{ Proof of Theorem \ref{th51}} \label{sec53}

We partition the archimedean arithmetic term $\Phi_{\infty}(n)$ as 
$$
\Phi_{\infty}(n) = \Phi_{\infty, 1}(n) + \Phi_{\infty, 2}(n),
$$
with   initial sum $\Phi_{\infty, 1}(n)$ defined by
\begin{equation} \label{601}
\Phi_{\infty, 1}(n)\, :=\quad  \sum_{k=1}^n \mu(k) \Phi^{*}\Big( \left\lfloor \frac{n}{k} \right\rfloor\Big) = \frac{1}{2} \sum_{k=1}^n \mu(k) \left\lfloor \frac{n}{k} \right\rfloor \Big( \left\lfloor \frac{n}{k} \right\rfloor +1\Big)
\end{equation}
and the second sum $\Phi_{\infty, 2}(n)$ defined by
\begin{equation} \label{601a}
\Phi_{\infty, 2}(n)\, := \quad \sum_{k=1}^{\Kn} \mu(k)\, \Big(\log (\G_{\lfloor n/k \rfloor})-\Phi^{*}\Big( \left\lfloor \frac{n}{k} \right\rfloor\Big) \Big), \quad\quad
\end{equation}
with $\Kn$ given by \eqref{Knn}.
We first derive an unconditional formula for the initial sum $\Phi_{\infty, 1}(n)$.

%%%%%%%%%%%%%%%%
% Theorem 5.2 
%%%%%%%%%%%%%%%
\begin{thm}\label{thA1}
Set  $\Phi(n) = \sum_{k=1}^n \varphi(k)$. Then one  has 
\begin{equation}\label{603}
\Phi(n) = \frac{1}{2} \sum_{k=1}^n \mu(k) \lfloor \frac{n}{k} \rfloor ( \lfloor \frac{n}{k}\rfloor +1),
\end{equation}
so that $\Phi(n) = \Phi_{\infty, 1}(n)$.
\end{thm}

\begin{proof}
We first show that 
\begin{equation}\label{602}
\Phi(n) =\frac{1}{2} \Big( \sum_{k=1}^n \mu(k) (\lfloor \frac{n}{k} \rfloor)^2  \Big)     +\frac{1}{2}.
\end{equation}
This equality is proved by induction on $n$; call its right side $S(n)$.
The extra term $\frac{1}{2}$ on the right side  is needed to establish  the base case $n=1$.
For the induction step, 
suppose  $S(n)= \Phi(n)$ for a given $n$. 
Since $\lfloor \frac{n+1}{k} \rfloor = \lfloor \frac{n}{k} \rfloor$ unless $k \mid (n+1)$, we have 
\begin{eqnarray*}
S(n+1)- S(n)  &=& \frac{1}{2} \sum_{d | (n+1)} \mu(d) \Big( \lfloor \frac{n+1}{d} \rfloor^2 - \lfloor \frac{n}{d} \rfloor^2\Big)\\
& = & \frac{1}{2} \sum_{d | (n+1)} \mu(d) \Big( 2 (\frac{n+1}{d}) -1 \Big)\\
& = & \sum_{d |  (n+1)} \mu(d) \frac{n+1}{d}  - \frac{1}{2} \left(\sum_{d|(n+1)} \mu(d)\right).\\
&=& \varphi(n+1),
\end{eqnarray*}
This shows $S(n+1) = \Phi(n+1),$ completing the induction step, proving \eqref{602}.

To establish \eqref{603}, comparing  the definition 
\eqref{601}
 of $\Phi_{\infty, 1}(n)$ with the right 
side of \eqref{602}, we obtain, 
$$
\Phi_{\infty, 1}(n)= \Phi(n)-\frac{1}{2}+\frac{1}{2}\sum_{k=1}^n \mu(k) \lfloor \frac{n}{k} \rfloor=\Phi(n) 
$$
where the last equality used \eqref{magic}.
\end{proof} 

%%%%%%%%%
% Remark 5.3
%%%%%%
\begin{rem}\label{rem62}
Combining \eqref{602}   with the known asymptotic for $\Phi(n)$  yields
$$
\frac{1}{2}\sum_{k=1}^n  \mu(k)(\lfloor\frac{n}{k}\rfloor)^2  = \frac{3} {\pi^2} n^2 + O ( n \log n).
$$
Here the  remainder term
$E(n) = \Phi(n)- \frac{3}{\pi^2} n^2$  is known to have large oscillations of magnitude 
at least $\Omega (n \sqrt{\log\log n})$ (see Section  \ref{sec32}).
One  can consider  a similar sum which does not apply  the fractional part  function, 
and obtain a similar unconditional estimate 
\begin{equation}\label{none}
\frac{1}{2} \Big( \sum_{k=1}^n \mu(k) ( \frac{n}{k})^2 \Big) = \frac{3} {\pi^2} n^2 + O ( n).
\end{equation}
Under the assumption of the Riemann hypothesis, one can establish  a much smaller error term
$$
\frac{1}{2} \Big( \sum_{k=1}^n \mu(k) ( \frac{n}{k})^2 \Big) = \frac{3} {\pi^2} n^2 + O ( n^{\frac{1}{2}+ \epsilon}).
$$
Comparing the right side of \eqref{602}  with  \eqref{none} reveals that 
the oscillations in the remainder term $E(n)$ are coming  from the application of the floor function
in  the sum \eqref{602}.
\end{rem}

We next derive  estimates for $\Phi_{\infty, 2}(n)$.

%%%%%%%%%%%%%%%%
% Theorem 5.4
%%%%%%%%%%%%%%%
\begin{thm}\label{thA2}
(1)  There holds unconditionally 
\begin{equation}\label{uncond1}
\Phi_{\infty, 2}(n) = \frac{n}{2}  \sum_{1 \le k \le  \sqrt{n}} \frac{\mu(k)}{k} \log k -
\frac{n}{2}\log \left(\frac{2\pi n}{e}\right)\sum_{1 \le k \le \sqrt{n} } \frac{\mu(k)}{k} 
+ O \left(\sqrt{n} \log n \right).
\end{equation}

 (2) Assuming the   Riemann hypothesis,  for each $\epsilon >0$ there holds 
\begin{equation}\label{condRH1}
\Phi_{\infty, 2}(n) = -\frac{1}{2} \psi(n)  + O \big( n^{3/4 + \epsilon} \big).
\end{equation}
where $\psi(x) := \sum_{n \le x} \Lambda(n)$.
% has $\psi(x) = x + O( x^{1/2+ \epsilon}).$
\end{thm}

\begin{proof}
(1) To prove (1) from Theorem \ref{th20} we have 
$$
\log (\G_{ \lfloor n/k\rfloor}) = \Phi^{\ast}\left( \left\lfloor \frac{n}{k}\right\rfloor\right)
%\frac{1}{2} \lfloor \frac{n}{k} \rfloor(\lfloor \frac{n}{k} \rfloor+1)
 - \frac{1}{2}\left\lfloor \frac{n}{k} \right\rfloor \log \left(\left\lfloor \frac{n}{k} \right\rfloor\right)
+\Big(\frac{1}{2}- \log(\sqrt{2 \pi}) \Big) \left\lfloor \frac{n}{k} \right\rfloor          + O \Big(  \log\Big(\left\lfloor \frac{n}{k} \right\rfloor \Big) \Big).
$$
We write $\lfloor n/k\rfloor = n/k - \{ n/k\}$ and obtain that 
\[\log (\G_{ \lfloor n/k\rfloor})- \Phi^{\ast}\left( \left\lfloor \frac{n}{k}\right\rfloor\right) \] 

can be written as 

\[- \frac{1}{2} \left(\frac{n}{k}- \Big\{ \frac{n}{k} \Big\}\right) \left(\log  \frac{n}{k}
 + \log \left(1 - \frac{k}{n} \Big\{ \frac{n}{k} \Big\}\right)+2\log\left(\frac{2\pi}{e}\right)\right)
 + O \left(  \log\left(\frac{n}{k} \right)\right).
\]

Using the estimate $\log(1 - \frac{k}{n} \{ \frac{n}{k} \}) = O ( \frac{k}{n})$ valid for $1 \le k \le \lfloor \sqrt{n} \rfloor$, 
and noting that in all cases $\lfloor \sqrt{n} \rfloor -1 \le \Kn \le \lfloor \sqrt{n} \rfloor$, we obtain
unconditionally 

\begin{equation}\label{eq-A2}
\Phi_{\infty, 2}(n) = - \frac{n}{2} \sum_{1 \le k \le \sqrt{n}} \frac{\mu(k)}{k} \log \frac{n}{k} 
- \log\left(\sqrt{\frac{2 \pi}{e}}\right) \sum_{1 \le k \le \sqrt{n}}  \mu(k) \,\frac{n}{k} 
+ O \Big( \sqrt{n} \log n \Big).
\end{equation}

Using $\log \frac{n}{k} = \log n - \log k$ in the first term, simplifying and collecting terms
yields \eqref{uncond1}.

(2) To prove (2), first, assuming the Riemann hypothesis, we have the estimate
\begin{equation}\label{eq-A22}
\sum_{ 1 \le k \le \sqrt{n} }  \frac{\mu(k)}{k}= O( n^{-1/4 + \epsilon}).
\end{equation}
To show this, we start from  the conditionally convergent sum  
$$\sum_{k=1}^{\infty} \frac{\mu(k)}{k} =0,$$
 a statement known to be equivalent to the Prime Number Theorem.
We then have 
$$\sum_{k=1}^N\frac{\mu(k)}{k} = - \sum_{k=N+1}^{\infty} \frac{ \mu(k)}{k}.$$
By partial summation, assuming RH, we obtain
\begin{eqnarray*}
\sum_{k=N+1}^{\infty}  \frac{\mu(k)}{k} &=& \sum_{k=N+1}^{\infty} ( M(k) - M(k-1) ) \frac{1}{k}\\
&=&  \frac{M(N)}{N+1} +  \sum_{k=N+1}^{\infty}  M(k) \big( \frac{1}{k} - \frac{1}{k+1} \big) \\
&=&  O \Big( N^{-1/2 +\epsilon} + \sum_{k=N+1}^{\infty} \frac { \, k^{1/2+\epsilon}}{ k(k+1)} \Big) = O \Big( N^{-1/2 +\epsilon}\Big).
\end{eqnarray*}
Choosing  $N=\sqrt{n}$ yields \eqref{eq-A22}.

Second, assuming the Riemann hypothesis, we have the estimate
\begin{equation}\label{eq-A23}
\sum_{1 \le k \le \sqrt{n}} \frac{\mu(k)}{k} \log k = -1+ O (n^{-\frac{1}{4} + \epsilon} \log n).
\end{equation}
To show that,   we start from  the conditionally convergent sum
$$
\sum_{k=1}^{\infty} \frac{\mu(k)}{k} \log k =\frac{d}{ds}( \frac{1}{\zeta(s)})|_{s=1} =  -1,
$$
again a result at the depth of the Prime Number Theorem. The result \eqref{eq-A23}
is proved by a similar partial summation argument to the above.

The estimate \eqref{eq-A22} allows us to  bound the second sum on the right in \eqref{uncond1} by $O( n^{3/4+\epsilon})$.
The estimate \eqref{eq-A23} allows us to estimate the first sum on the right in \eqref{uncond1} by $-\frac{1}{2}n + O( n^{3/4+\epsilon})$.
In consequence, the RH yields
$$
\Phi_{\infty, 2}(n) = -\frac{1}{2} n + O (n^{3/4+ \epsilon}).
$$

Third, the Riemann hypothesis  is well known to be equivalent to the assertion
$$
\psi(n) = n + O \Big( n^{1/2} (\log n)^2 \Big).
$$
This fact proves (2).
\end{proof}

%%%%%%%%%%%%%%%%
% Proof of Theorem 5.1
%%%%%%%%%%%%%%%%%%
\begin{proof}[Proof of Theorem \ref{th51}.]
We assume that the Riemann hypothesis holds.
On combining Theorem \ref{thA1} with Theorem  \ref{thA2} (2),
we  obtain
$$
\Phi_{\infty}(n) = \Phi_{\infty, 1}(n) + \Phi_{\infty, 2}(n) = \Phi(n) - \frac{1}{2} \psi(n) + O( n^{\frac{3}{4}+\epsilon}).
$$
Combining this estimate with Mikolas's  Theorem \ref{th43}
gives the estimate
$$
\log(\F_n) = \Phi_{\infty}(n) + O(n^{3/4+\epsilon}),
$$ 
as desired.
\end{proof}

%%%%%
% section 6
%%%%%
%************************************************************************
%
%  section 6
%
%
%************************************************************************

\section{Is there an $\ord_p(\F_n)$ analogue of the Riemann hypothesis?}\label{sec6}

The problem of determining the behavior of the
functions $\ord_p(\F_n)$ for a fixed prime $p$ may
be a difficult one, because the analogous problem at the real place encodes the
Riemann hypothesis,  in the form Theorem \ref{th43} (2).  
One may ask more: {\em Is it possible to  encode the Riemann hypothesis itself at
a single prime $p$, in terms of  the
behavior of $\W_p(\F_n) =\ord_p(\F_n)$ as $n \to \infty$?}

In Section \ref{sec5} we reformulated the Riemann hypothesis entirely in terms of 
the sizes $\log (\F_n)$ and $\log(\G_n)$ of Farey products
and unreduced Farey products, respectively.  
The advantage of this reformulation is that has formal analogues  defined for each finite prime $p$.
On the left side, the quantity to approximate,
$\log(\F_n)$,  has an analogue quantity  defined for each prime,  $\ord_p(\F_n)$.
On the right side, the  new arithmetic main term $\Phi_{\infty}(n)$ we introduced  has analogue
quantities  built out
of replacing the quantities $\log(\G_k)$ with $\ord_p(\G_k)$ 
in suitable ways. 

The new arithmetic main terms that we introduce this way are necessarily  arithmetic functions
exhibiting oscillations, because $\ord_p(\F_n)$ exhibits oscillations and sign changes. 
These terms contain new  kinds of arithmetic information which may be of  interest in their own right,
encoded as new sorts of arithmetic sums mixing the M\"{o}bius function with base $p$ radix expansion data.
We  will see there is more than one 
possible choice to consider for these ``main terms" for a finite prime $p$. 
With each choice we have an associated remainder term, and we  study these remainder terms  experimentally.

In parallel with the archimedean case we expect the Riemann hypothesis to manifest itself in bounds on 
the size of remainder terms. 
We  present below  computational results  that
suggest such a formulation may be possible.

%************************************************************************
%
%  section 6.1
%
%
%************************************************************************

\subsection{Arithmetic main terms and remainder terms for finite primes $p$ }\label{sec61}

We now formulate ``arithmetic main terms" for  $\ord_p(\F_n)$.
For each prime $p$  we can define by analogy a decomposition
\begin{equation}\label{non-arch2}
\W_p(\F_n) 
%:=\ord_p (\F_n)
 = \Phi_{p}(n) + \R_{p}(n).
\end{equation}
by making a suitable choice of a  $p$-adic arithmetic term. It is not  clear a priori whether
there should be included an analogue of the first term on the right side of \eqref{arch-mainterm} or not. 
We therefore experimentally investigate  three plausible choices for the arithmetic term, denoting them $ \Phi_{p,j}(n)$ for  $0 \le j \le 2$, 
in which  we may or may not choose to include a correction term of quantities summed over the whole interval
$1 \le k \le n$. We recall  the formula 
$$\ord_p(\G_n) = \frac{2}{p+1}S_p(n) - \frac{n-1}{p+1} d_p(n)$$
 given
in Theorem \ref{th39}, which splits $\ord_p(\G_n)$ into a smooth term and an oscillatory term, respectively.
We consider the options whether to remove
none or one of the two
sums on the right side {\em over the whole interval $1 \le k \le n$.}

The three  options are  first,  to have no correction term,  
\begin{equation}\label{p-mainterm31}
 \Phi_{p,0}(n) := 
\sum_{k=1}^{\Kn} {} \mu(k)\, 
\Big( \ord_p (\G_{\lfloor n/k \rfloor})    \Big),
\end{equation}
or second, to add a correction term that removes the contribution of the $d_p(n)$,
\begin{equation}\label{p-mainterm3}
\Phi_{p,1}(n) :=  \Phi_{p,0}(n)
-\frac{n-1}{p-1}\left(\sum_{k=1+\Kn}^n \mu(k) d_p\left(\Big\lfloor \frac{n}{k} \Big\rfloor \right)\right),
\end{equation}
or  third, to have a correction term that removes  the contribution of the $S_p(n)$, 
\begin{equation}\label{p-mainterm3}
\Phi_{p,2}(n) :=  \Phi_{p,0}(n) - \frac{2}{p-1}\Big(\sum_{k=\Kn+1}^n \mu(k) S_p(\lfloor \frac{n}{k} \rfloor)\Big).
\end{equation}
In each case  the remainder term $\R_{p,j}(n)$ is defined by \eqref{non-arch2}
for $j=0, 1, 2$ with the $K_n$ is as defined in \eqref{Knn}.
The remainder terms   $\R_{p,j}(n)$in the three cases  are explicitly given  by
\begin{equation*}\label{p-remterm31}
\R_{p, 0}(n) \quad  = \quad \sum_{\ell = 1}^{ \Ln} \Big( M\left(\frac{n}{\ell}\right) - M(  n/(\ell+1) \Big)\, 
\big( \ord_p (G_{\ell})  \big), \quad\quad\quad\quad
\end{equation*}
\begin{eqnarray*}\label{p-remterm33}
\R_{p,1}(n)  &= &\sum_{\ell = 1}^{\Ln} \Big( M\left(\frac{n}{\ell}\right) - M\Big(  \frac{n}{\ell+1}\Big)\Big)\, 
\big( \ord_p (G_{\ell}) + \frac{\ell-1}{p-1} d_p(\ell) \big) \nonumber \\
& =  & \sum_{\ell = 1}^{ \Ln} \Big( M\left(\frac{n}{\ell}\right) -M\Big(  \frac{n}{\ell+1}\Big) \Big)\, 
\Big( \frac{2}{p-1} S_p(\ell) \Big),
\end{eqnarray*}
\begin{eqnarray*}\label{p-remterm32}
\R_{p,2}(n) &= &\sum_{\ell = 1}^{ \Ln}  \Big( M\left(\frac{n}{\ell}\right) - M\Big(  \frac{n}{\ell+1}\Big)\Big)\, 
\big( \ord_p (G_{\ell}) - \frac{2}{p-1} S_p(\ell) \big)     \nonumber \\
& = & - \sum_{\ell = 1}^{ \Ln}  \Big( M\left(\frac{n}{\ell}\right) - M\Big(  \frac{n}{\ell+1}\Big)\, 
\Big( \frac{\ell-1}{p-1} d_p(\ell) \Big).
\end{eqnarray*}

With  these definitions we have the  identity
\begin{equation}\label{identity1}
\R_{p,0}(n) = \R_{p, 1}(n) + \R_{p, 2}(n).
\end{equation}
For our calculations we  choose $\Ln = \lfloor \sqrt{n} \rfloor$ as above.

The formulas for $\Phi_{p, j}(n)$ embody arithmetic sums of  new types,
which involve  M\"{o}bius function values multiplied against base $p$ radix
expansion data of $k$ with $1 \le k \le n$.

%************************************************************************
%
%  section 6.2
%
%
%************************************************************************

\subsection{Remainder terms for $p=3$: experimental data}\label{sec62}
The following figures give data for $p=3$ for these three choices of remainder terms $\R_{p,j}(n).$

%\newpage
%%%%%%%%%%%%%%
%Figure 6.1 - R_3(n)plot
%%%%%%%%%%%%%%%

\begin{figure}[!htb]
\includegraphics[width=125mm]{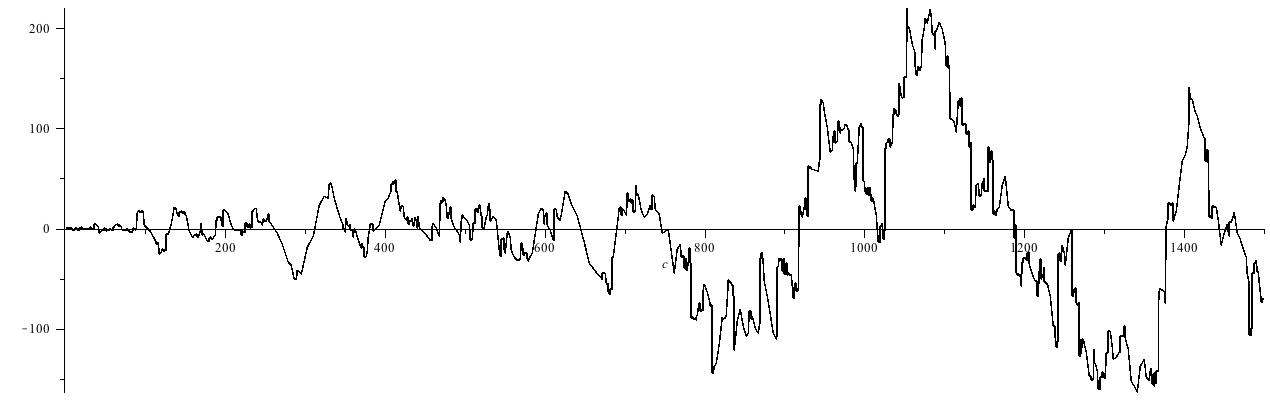}
%ord2fnlinear.jpg}
\caption{$\R_{3,0}(n)$, $1 \le n \le 1500$.}
%\label{fig61}
\label{fig52-R3}
\end{figure}

%%%%%%%%%%%%%%%%%%%%%%%%%%%
%In the definition of the $p$-adic  arithmetic term \eqref{p-mainterm3} given above we do not include any 
%correction term, i.e.  we  do not subtract off any part of the expression for $\ord_p(G_n)$
%over the full range $1 \le k \le n$. Candidates for such a correction term might be either
%the smooth part $\frac{2}{p-1}S_p(n)$ of the formula for $\ord_p(\G_n)$ in 
%Theorem \ref{th39}, or alternatively, its complementary term,
%the fluctuating term $\frac{n+1}{p-1}\dgt_p(n)$.
%%%%%%%%%%%%%%%%%%%%%%%%%%%%

%%%%%%%%%%%%%%
%Figure 6.2 - R_3(n)plot
%%%%%%%%%%%%%%%

\begin{figure}[!htb]
\includegraphics[width=125mm]{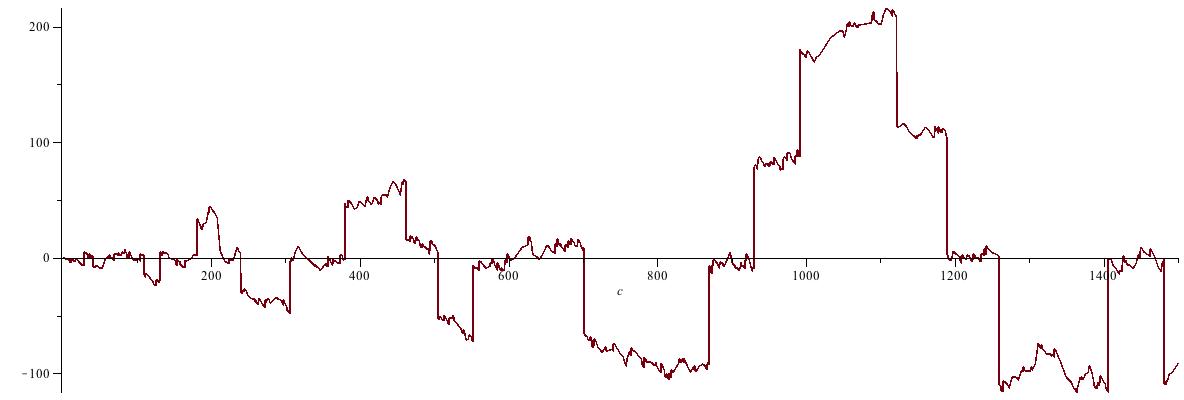}
%psum2_p=3.jpg}
%ord2fnlinear.jpg}
\caption{$\R_{3,1}(n)$, $1 \le n \le 1500$.}
\label{fig53-R3sum2}
\end{figure}

%%%%%%%%%%%%%%
%Figure 6.3 - R_3(n)plot
%%%%%%%%%%%%%%%

\begin{figure}[!htb]
\includegraphics[width=125mm]{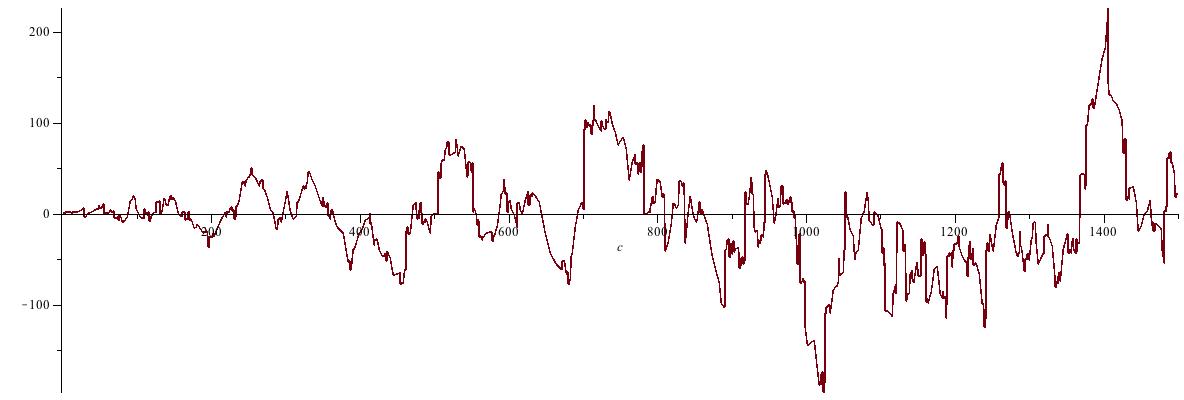}
%{psum1_p=3.jpg}
%ord2fnlinear.jpg}
\caption{$\R_{3,2}(n)$, $1 \le n \le 1500$.}
\label{fig54-sum3}
\end{figure}

In these plots all three remainder terms seem roughly the same size; this size however
is slightly larger in magnitude than that seen for $\R_{\infty}(n)$. The identity \eqref{identity1} implies that 
either all three sums are of the same order of magnitude, or else one sum is significantly smaller than the other
two.  

%%%%%%%%%%%%%%%%%%%%%%%%%%
%(Note that Figure \ref{fig54-sum3} plots $-\R_{3,2}(n)$ above and not $\R_{3,2}(n)$.)
%The term  $(2/(p-1)\Sum_p( \lfloor n/k \rfloor)$ is the smooth part of the expression  for $\ord_p(\G_n)$,
%according to Delange's theorem \ref{th313}. 
%%%%%%%%%%%%%%%%%%%%%%%%%%%

We observe the surprising feature  that  the graph of $\R_{3, 1}(n)$ in Figure \ref{fig53-R3sum2}  
(more precisely  of its negative $-\R_{3, 1}(n)$) has a striking qualitative resemblance to the remainder term $\R_{\infty}(n)$.
It has  large abrupt jumps and some relatively flat spots, with jumps at  exactly the same
points as for $\R_{\infty}(n)$; the jumps appear to be  larger than that of $R_{\infty}(n)$ by a factor of roughly $\frac{5}{3}$.
We  found that similar  qualitative behavior occurs for $-\R_2(n)$ and $-\R_5(n)$ over the same range,
 with identical jump locations and multiplicative scaling factors of jump sizes roughly $3$ and $\frac{5}{4}$, respectively.

%************************************************************************
%
%  section 6.3
%
%
%************************************************************************

\subsection{Remainder term growth rates: hypotheses}\label{sec63}
On the strength of the empirical  observations above , we formulate for consideration
the following hypotheses.\medskip

\noindent {\bf Hypothesis $\R_{p,1}$.} {\em For each fixed $\epsilon >0$ there holds, as $n \to \infty$,}
\begin{equation} 
\R_{p,1}(n) = O( n^{\frac{3}{4} + \epsilon}).
\end{equation}

The similarity of the shape and magnitude of the plot of 
the remainder term $\R_{3, 1}(n)$  to  that of $\R_{\infty}(n)$,
including the jump sizes, is  striking. The structure and location of
the jumps is explainable as an artifact the hyperbola method;
the jumps are at $n=m(m+1)$ with $m$ squarefree and the jump directions are $-\mu(m)$.
The hypothesis above concerns the
growth rate of the reminder term and not its appearance,
and one may  ask whether this growth rate might be related to the Riemann hypothesis.
% the  possibility that this statistic might be 
%directly related to the Riemann hypothesis.

Since  the plots of all three of the $\R_{3, j}(n)$ above empirically appear to be about  the same size,
we also propose  for consideration: \medskip

\noindent {\bf Hypothesis $\R_{p,2}$.} {\em For each fixed $\epsilon >0$ there holds, as $n \to \infty$,}
\begin{equation} 
\R_{p,2}(n) = O( n^{\frac{3}{4} + \epsilon}).
\end{equation}

We have no  theoretical evidence supporting  Hypothesis $\R_{p, 2}$, but we have
checked it empirically for other small primes, on limited data sets.
  We speculate that  Hypothesis $\R_{p, 2}$, if true,  
might encode arithmetic
data  specific to the prime $p$, directly relating the M\"{o}bius function and the base $p$
expansions of integers, not  necessarily related to the Riemann hypothesis.

Besides Hypothesis $\R_{p,1}$ and $\R_{p,2}$, 
one may formulate in parallel a third hypothesis.\medskip

\noindent {\bf Hypothesis $\R_{p,0}$,}
{\em For each fixed $\epsilon >0$ there holds, as $n \to \infty$,}
\begin{equation} 
\R_{p,0}(n) = O( n^{\frac{3}{4} + \epsilon}).
\end{equation}

The  additive identity 
\eqref{identity1} relating the $\R_{p,j}$  for  $0 \le j \le 2$ above shows that the truth of 
any two of these hypotheses would imply the truth of the third. We have no independent
theoretical evidence supporting  Hypothesis $R_{p, 0}$.

%%%%%%%%%%%%%%%%%%%%%%%%%%%%
%The structure of the graph of $\R_{\infty}$ and of $\R_{p,1}$ leading to their  
%which  involves  the Riemann zeta function and zeta zeros.
%%%%%%%%%%%%%%%%%%%%%%%%%%%%%%

%************************************************************************
%
%  section 7
%
%
%************************************************************************

\section{Concluding Remarks: Arithmetic encodings of the  Riemann hypothesis}\label{sec7}

To summarize our experimental work in Section \ref{sec5} and \ref{sec6} , we have found:
\begin{enumerate}
\item[(1)] The remainder term $\R_{\infty}(n)$ plotted in Figure \ref{fig51-R-inf} is provably related to the Riemann hypothesis
(via Theorem \ref{th51})  and its plots reveal 
a striking internal structure of jumps worthy of further investigation.
\item[(2)]
 The plot  for $p=3$ of $\R_{3,1}(n)$ 
pictured in Figure \ref{fig53-R3sum2} exhibits a similar internal structure to $\R_{\infty}(n)$,  
which implies nearly perfect  correlation of  the statistic $\R_{3, 1}(n)$  with  $\R_{\infty}(n).$
Similar internal structure was found  in plots for $p=2$ and $p=5$ (not pictured).
\end{enumerate}
The observation (2) was  surprising,  
in that the quantities defining 
the statistic $\R_{p,1}(n)$ seemed  very  different from  those defining  $\R_{\infty}(n)$.
Subsequent investigation revealed that 
the main features in these plots, with their pattern of large jumps followed by 
slow variation, can be explained  as being an artifact of the ``hyperbola method" truncation.
 The jumps are located at points $n=m(m+1)$
where $m$ is squarefree,  and the sign of the jumps is related to  $\mu(m)$. This direct
connection of the error term with the M\"{o}bius function indicates that the zeta zeros influence at
least part of its behavior.
%of the remainder term. 
The  Riemann hypothesis may possibly be  encoded in the growth
rates of the remainder terms; this topic is left for further investigation. 
Our data are insufficient to give a reliable guess  on this  growth rate. 
The data obtained  is at least consistent  with the possibility 
that  the  Riemann hypothesis may be directly visible in the
growth rate of the remainder term statistics of $\ord_p(F_n)$ at a fixed finite prime $p$.
Larger scale computations are needed to confirm or disconfirm the possible $O(n^{\frac{3}{4} +\epsilon})$ behavior  of 
this remainder term.

\subsection*{Acknowledgments}
We  thank J. Arias de Reyna, R. C. Vaughan and the two reviewers for helpful comments on this paper.
The first author thanks Harm Derksen for  bringing up questions on Farey products, resulting in
 \cite{DL11a}, \cite{DL11b}. 
Work of H. Mehta on this project started as part of  an REU program at  the University of Michigan,  
with the first author as mentor. \bigskip \bigskip

%The work of  J. Lagarias was partially supported by NSF grants 
%DMS-0801029 and DMS-1101373.

%************************************************************************
%
%  section Appendix A
%
%
%************************************************************************

\section*{Appendix A: Empirical Results for $p=3$} \label{secAA}

This Appendix presents plots and tables for 
%$\ord_p(\G_n)$and 
$\ord_p(\F_n)$ for $p=3$, supplementing  the data for $p=2$
given in graphs and tables in Section \ref{sec43}.
%We present a plot of $\ord_3(\F_n)$ for $1 \le 1214 = 3^6 + 2 \cdot 3^5 -1.$
%This data may be compared with Figure \ref{fig41-ord2} for $p=2$.

\newpage
Figure A.1 plots the values of $\ord_4(\F_n)$ for $1 \le n < 1214$.
The cutoff value for this table is not a power of $3$, since $3^6=729$ and $3^7= 2187$.
It was chosen to be roughly $1/2 3^7$, the same size as the cutoff value for powers of $2$ for the graph
in Section \ref{sec43}.
%%%%%%%%%%%%%
% Figure A.3
%%%%%%%%%%%%%%
% Ord_3 F_n
%%%%%%%%%%%%%%%

\begin{figure}[!htb]
\includegraphics[width=130mm] {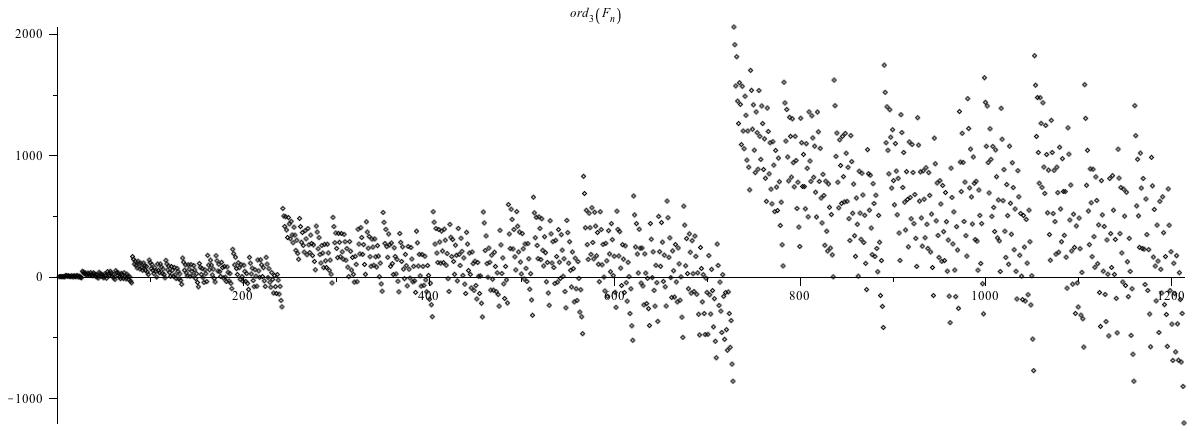}
\caption*{Figure A.1. $\ord_3(\F_n)$, $1 \le n \le 1214$.}
\label{figA1-ord3}
\end{figure}

 Table A.1  presents  data on the jump for $\ord_p(\F_{p^k-1})$ to
$\ord_p(\F_{p^k})$ for the prime $p=3$.  This data may be compared with
Table 4.1 for $p=2$.\\

%***********************************************************************
%
% TABLE A. 1.2 --
%
%************************************************************************

\begin{minipage}{\linewidth}
\begin{center}
\begin{tabular}{|r | r | r | c |c|}
\hline
%Set &
 \mbox{Power $r$} & $N= 3^r -1$ & $\ord_3(\F_{3^r -1})$  &  $-\frac{1}{N} \ord_3(F_{3^r -1})$  & $-\frac{1}{N \log_3 N} \ord_3(F_{3^r -1})$ \\
%\multicolumn{2}{c}{$\alpha_k$}{c{Hausdorff dim}  \\
\hline
 $1$ &   $2$& $0$ &  $0.0000$  &  $0.0000$\\
$2$ &   $8$ & $-1$ & $ 0.1250$  &  $0.0538$\\
$3$ &  $26$ & $-9$  &   $0.3461$  & $0.1167$\\ 
$4$ & $80$ &  $-50$  & $0.6250$   &   $0.1567$\\ 
$5$ &  $242$ & $-248$  & $1.0248$  &  $ 0.2051$\\ 
$6$ & $728$ &   $-860$    & $1.1813$  & $0.1969 $\\
$7$ & $2186$ & $-3333$ & $1.5247$  & $0.2178$\\
$8$ &  $6560$ & $-12380$ & $1.8872$  & $0.2359$\\ 
$9$ &  $19682$ & $-45773$ & $2.3256$  & $ 0.2584$\\ 
$10$ & $59048$ & $-148338$ & $2.5122$ & $0.2512$ \\ \hline
% $11$ &  $2047$ &  $-3830$ &  $1.8710$ &  $0.1701$\\
%$12$ &  $4095$ &  $-7352$ & $ 1.7953$ &  $ 0.1496$ \\
%$13$ &  $ 8191$ & $-20348$  &   $2.4842$ &  $0.1910$\\ 
%$14$ &  $16383$ & $-41750$  & $2.5484$  & $0.1820$ \\ 
%$15$ &  $ 32767$ & $-89956$  & $2.7453$  &  $0.1830$ \\ \hline
\end{tabular} \par
\bigskip
\hskip 0.5in {\rm TABLE A.1.}  
{\em Values at $N= 3^r -1$ of $\ord_3(\F_N)$.}
%\label{Table41}
%and $p=3$
\newline
\newline
\end{center}
\end{minipage}

%************************************************************************
%
%  Bibliography 
%
%
%************************************************************************

\end{document}